\ifx \processToSlides \undefined
  \ifx\usepackage\RequirePackage    
    \let\usingAmsArtXII\usepackage  
  \fi
\else
  \documentclass[12pt]{amsart}
  \usepackage{amssymb,amscd}
  \usepackage[letterpaper,margin=1.2cm,includeheadfoot]{geometry}
  \def \useHugeSize {}
  \def \numberingIsThrough {}
\fi

\ifx \labelsONmargin\undefined

  \ifx\RequirePackage\undefined
    \expandafter\ifx\csname amsart.sty\endcsname\relax
        \documentstyle[12pt,amssymb,amscd]{amsart}
    \fi

    \def\atSign{@@}

    \def\mathbb{\Bbb}
    \def\mathfrak{\frak}
    \def\mathbf{\bold}
    \ifx\boldsymbol\undefined   
      \def\boldsymbol#1{{\bold #1}}
    \fi

    \def\mathbit{\boldsymbol}

    \makeatletter
    \newenvironment{proof}{%
         \@ifnextchar[{%
                       \expandafter\let\expandafter\end@proof
                         \csname endpf*\endcsname
                         \my@proof
                      }{\let\end@proof\endpf\pf}%
        }{\end@proof}
    \def\my@proof[#1]{\@nameuse{pf*}{#1}}
    \makeatother

    \def\xrightarrow[#1]#2{@>{#2}>{#1}>}
    \def\xleftarrow[#1]#2{@<{#2}<{#1}<}

    \def\providecommand#1{\def#1}
    \def\emph#1{{\em #1}}
    \def\textbf#1{{\bf #1}}

    \def\mathring{\overset{\,\,{}_\circ}}
  \else
      \ifx\usepackage\RequirePackage
      \else
        \documentclass[12pt]{amsart}
        \usepackage{amssymb,amscd}
    \let\usingAmsArtXII\usepackage
      \fi
      \ifx \mathring\undefined      
        \DeclareMathAccent{\mathring}{\mathalpha}{operators}{"17}
      \fi

      \long\def\FAKEendPROOF{\endtrivlist}
      \ifx\endproof\FAKEendPROOF
      \def\endproof{\qed\endtrivlist}
      \fi

      \ifx\mathbit\undefined    
        \DeclareMathAlphabet{\mathbit}{OML}{cmm}{b}{it}
      \fi

      \def\atSign{@}

      \ifx \usingAmsArtXII \undefined
      \else             
    \advance\oddsidemargin -0.1\textwidth
    \evensidemargin\oddsidemargin
      \fi

      \def\Sb#1\endSb{_{\substack{#1}}}
      \def\Sp#1\endSp{^{\substack{#1}}}

  \fi
\makeatletter
\@ifundefined{DeclareFontShape}
     {\@ifundefined{selectfont}
             {
                \message{We need AmS-LaTeX, this version of LaTeX does not
                        support it}\@@end
             }{
                \def\mathcal{\cal}
                
                \def\pcyr{%
                        \def\default@family{UWCyr}%
                        \let\oldSl@\sl
                        \def\sl{\def\default@shape{it}\oldSl@}%
                        \cyracc
                        \language\Russian\family{UWCyr}\selectfont
                }
             }}{
                \DeclareFontEncoding{OT2}{}{\noaccents@}
                \DeclareFontSubstitution{OT2}{cmr}{m}{n}
                \DeclareFontFamily{OT2}{cmr}{\hyphenchar\font45 }
                \DeclareFontShape{OT2}{cmr}{m}{n}{%
                     <5><6><7><8><9><10>gen*wncyr %
                     <10.95><12><14.4><17.28><20.74><24.88> wncyr10 %
                }{}
                \DeclareFontShape{OT2}{cmr}{m}{it}{%
                     <5><6><7><8><9><10> gen * wncyi%
                     <10.95><12><14.4><17.28><20.74><24.88> wncyi10%
                }{}
                \DeclareFontShape{OT2}{cmr}{bx}{n}{%
                     <5><6><7><8><9><10> gen * wncyb%
                     <10.95><12><14.4><17.28><20.74><24.88> wncyb10%
                }{}
                \DeclareFontShape{OT2}{cmr}{m}{sl}{%
                     <-> ssub * cmr/m/it%
                }{}
                \DeclareFontShape{OT2}{cmr}{m}{sc}{%
                     <5><6><7><8><9><10>%
                     <10.95><12><14.4><17.28><20.74><24.88> wncysc10%
                }{}
                \DeclareFontFamily{OT2}{cmss}{\hyphenchar\font45 }
                \DeclareFontShape{OT2}{cmss}{m}{n}{%
                     <8><9><10> gen * wncyss%
                     <10.95><12><14.4><17.28><20.74><24.88> wncyss10%
                }{}
                
                \def\cyrencodingdefault{OT2}
                \def\pcyr{%
                        \cyracc
                        \let\encodingdefault\cyrencodingdefault
                        \language\Russian\fontencoding{OT2}\selectfont
                }
             }

\@ifundefined{theorembodyfont}
     {
        \def\theorembodyfont#1{\relax}
        \makeatletter
          \let\@@th@plain\th@plain
          \def\th@plain{ \@@th@plain \slshape }
        \makeatother
        \let\normalshape\relax
     }{}

\ifx\cprime\undefined
     \def\cprime{$'$}
\fi
\makeatletter
  \def\@sect@my#1#2#3#4#5#6[#7]#8{%
\ifnum #2>\c@secnumdepth
   \let\@svsec\@empty
 \else
   \refstepcounter{#1}%
\edef\@svsec{\ifnum#2<\@m
             \@ifundefined{#1name}{}{\csname #1name\endcsname\ }\fi
\noexpand\rom{\csname the#1\endcsname.}\enspace}\fi
 \@tempskipa #5\relax
 \ifdim \@tempskipa>\z@ 
   \begingroup #6\relax
   \@hangfrom{\hskip #3\relax\@svsec}{\interlinepenalty\@M #8\par}%
   \endgroup
   \if@article\else\csname #1mark\endcsname{%
        \ifnum \c@secnumdepth >#2\relax\csname the#1\endcsname. \fi#7}\fi
\ifnum#2>\@m \else
       \let\@tempf\\ \def\\{\protect\\}\addcontentsline{toc}{#1}%
{\ifnum #2>\c@secnumdepth \else
             \protect\numberline{%
               \ifnum#2<\@m
               \@ifundefined{#1name}{}{\csname #1name\endcsname\ }\fi
               \csname the#1\endcsname.}\fi
           #8}\let\\\@tempf
     \fi
 \else
  \def\@svsechd{#6\hskip #3\@svsec
    \@ifnotempty{#8}{\ignorespaces#8\unskip
       \ifnum\spacefactor<1001.\fi}%
        \ifnum#2>\@m \else
          \let\@tempf\\ \def\\{\protect\\}\addcontentsline{toc}{#1}%
            {\ifnum #2>\c@secnumdepth \else
              \protect\numberline{%
                \ifnum#2<\@m
                \@ifundefined{#1name}{}{\csname #1name\endcsname\ }\fi
                \csname the#1\endcsname.}\fi
             #8}\let\\\@tempf\fi}%
 \fi
\@xsect{#5}}
  \ifx\RequirePackage\undefined
  \let\@sect\@sect@my             
  \fi
  \ifx\RequirePackage\undefined 
  \def\th@remark@my{\theorempreskipamount6\p@\@plus6\p@
    \theorempostskipamount\theorempreskipamount
    \def\theorem@headerfont{\it}\normalshape}
    \let\th@remark\th@remark@my
  \else             
    \let\o@@remark\th@remark

    \ifx\theorempostskipamount\undefined        
    \else                       
      \def\th@remark{\o@@remark
    \ifdim\theorempostskipamount < 2pt\relax
      \theorempostskipamount\theorempreskipamount
         \multiply\theorempostskipamount\tw@
         \divide\theorempostskipamount\thr@@
    \fi
      }
    \fi
  \fi
\let\myLabel\@gobble
\def\labelsONmargin{\@mparswitchfalse\def\myLabel##1{\@bsphack\marginpar
                                  {\normalshape\tiny\rm Label ##1}\@esphack}}
\ifx \url\undefined
  \def\url#1{{\tt #1}}%
\fi
\def\cyracc{\def\u##1{
                \if \i##1\char"1A%
                \else \if I##1\char"12%
                \else \accent"24 ##1\fi\fi }%
\def\"##1{\if e##1{\char"1B}%
                \else \if E##1{\char"13}%
                \else \accent"7F ##1\fi\fi }%
\def\9##1{\if##1z\char"19
\else\if##1Z\char"11
\else\if##1E\char"03
\else\if##1e\char"0B
\else\if##1u\char"18
\else\if##1U\char"10
\else\if##1A\char"17
\else\if##1a\char"1F
\else\if##1p\char"7E
\else\if##1P\char"5E
\else\if##1Q\char"5F
\else\if##1q\char"7F
\else\if##1i\char"1A
\else\if##1I\char"12
\else\if##1N\char"7D
\fi
\fi
\fi
\fi
\fi
\fi
\fi
\fi
\fi
\fi
\fi
\fi
\fi
\fi
\fi
}%
\def\cydot{{\kern0pt}}}%
\def\cydot{$\cdot$}

\ifx\Russian\undefined
        \def\Russian{0\relax
    \message{Don't know the hyphenation rules for Russian^^J
                        Please do INITeX with `input  russhyph' in the
                        command line}%
                \gdef\Russian{0\relax}%
        }
\fi
\ifx\RequirePackage\undefined           
  \def\@putname#1#2#3#4{\def\@@ref{#3}\let\old@bf\bf
        \def\bf##1{\old@bf\if?\noexpand##1?{#4}\else##1\fi}%
    #1{#2}%
        \let\bf\old@bf}
\else                       
  \def\@putname#1#2#3#4{\def\@@ref{#3}\let\old@bf\bf    
    \let\old@reset@font\reset@font          
        \def\bf##1{\old@bf\if?\noexpand##1?{#4}\else##1\fi}%
    \def\reset@font##1##2{\old@reset@font##1\if?\noexpand##2?{#4}\else##2\fi}#1{#2}%
        \let\bf\old@bf\let\reset@font\old@reset@font}
\fi
\let\my@ref=\ref
\def\ref#1{\@putname\my@ref{#1}{#1}{\tiny\rm\@@ref}}
\let\my@pageref=\pageref
\def\pageref#1{\@putname\my@pageref{#1}{#1}{\tiny\rm\@@ref}}
\let\my@cite=\cite
\def\cite#1{\@putname\my@cite{#1}{\@citeb}{\tiny\rm\@@ref}}
\makeatother
\fi
\relax
\theoremstyle{plain} 
\theorembodyfont{\sl}

\ifx\useDoubleSpacing\undefined
\else
  \usepackage{setspace}
\fi

\ifx \numberingIsThrough \undefined
\numberwithin{equation}{section}
\fi
\theorembodyfont{\rm}
\theoremstyle{definition}

\ifx \numberingIsThrough \undefined
\newtheorem{definition}{Definition}[section]
\else
\newtheorem{definition}{Definition}
\fi

\theoremstyle{remark}

\newtheorem{remark}[definition]{Remark} 

\ifx \numberingIsThrough \undefined
\newtheorem{note}{Note}[section] 
\newtheorem{summary}{Summary}[section] 
\else

\fi

\theoremstyle{plain} 
\theorembodyfont{\sl}

\newtheorem{theorem}[definition]{Theorem}
\newtheorem{lemma}[definition]{Lemma}
\newtheorem{corollary}[definition]{Corollary}

\begin{document}
\bibliographystyle{amsplain}

\ifx\useHugeSize\undefined
\else
\Huge
\fi

\relax
\renewcommand{\v}{\varepsilon} \newcommand{\p}{\rho}
\newcommand{\m}{\mu}
\def\im{{\bf im}}
\def\ker{{\bf ker}}
\def\Pic{{\bf Pic}}
\def\re{{\bf re}}
\def\e{{\bf e}}
\def\a{\alpha}
\def\ve{\varepsilon}
\def\b{\beta}
\def\D{\Delta}
\def\d{\delta}
\def\f{{\varphi}}
\def\ga{{\gamma}}
\def\L{\Lambda}
\def\lo{{\bf l}}
\def\s{{\bf s}}
\def\cB{{\mathcal {B}}}
\def\C{{\mathbb C}}
\def\Q{{\mathbb Q}}
\def\F{{\bf F}}
\def\G{{\mathfrak {G}}}
\def\g{{\mathfrak {g}}}
\def\b{{\mathfrak {b}}}
\def\q{{\mathfrak {q}}}
\def\f{{\mathfrak {f}}}
\def\k{{\mathfrak {k}}}
\def\l{{\mathfrak {l}}}
\def\m{{\mathfrak {m}}}
\def\n{{\mathfrak {n}}}
\def\o{{\mathfrak {o}}}
\def\p{{\mathfrak {p}}}
\def\s{{\mathfrak {s}}}
\def\t{{\mathfrak {t}}}
\def\r{{\mathfrak {r}}}
\def\z{{\mathfrak {z}}}
\def\h{{\mathfrak {h}}}
\def\H{{\mathcal {H}}}
\def\O{\Omega}
\def\A{{\mathcal {A}}}
\def\B{{\mathcal {B}}}
\def\M{{\mathcal {M}}}
\def\T{{\mathcal {T}}}
\def\N{{\mathcal {N}}}
\def\U{{\mathcal {U}}}
\def\Z{{\mathbb Z}}
\def\P{{\mathcal {P}}}
\def\GVM{ GVM }
\def\iff{ if and only if  }
\def\add{{\rm add}}
\def\ld{\ldots}
\def\vd{\vdots}
\def\sl{{\rm sl}}
\def\mod{{\rm mod}}
\def\len{{\rm len}}
\def\cd{\cdot}
\def\dd{\ddots}
\def\q{\quad}
\def\qq{\qquad}
\def\ol{\overline}
\def\tl{\tilde}
\def\nn{\nonumber}

\font\pic=eurm10
\def\d{\hbox{\pic\char'0144}}

\font\block=msbm10
\def\QQ{\hbox{\block\char'0121}}

\title[On the finite $W$-algebra for the Lie superalgebra $Q(n)$]{{On the finite $W$-algebra for the Lie superalgebra $Q(n)$}
{in the non-regular case}}

\author{ Elena Poletaeva and Vera Serganova}


\address{School of Mathematical and Statistical Sciences, University of Texas Rio Grande
Valley, Edinburg, TX 78539} \email{elena.poletaeva\atSign{}utrgv.edu}
\address{ Dept. of Mathematics, University of California at Berkeley,
Berkeley, CA 94720 } \email{serganov\atSign{}math.berkeley.edu}

\maketitle

\begin{abstract} In this paper we study the finite $W$-algebra for the queer Lie superalgebra $Q(n)$ associated with the non-regular even nilpotent coadjoint orbits in the case when the corresponding nilpotent element has Jordan blocks each of size $l$. We prove that this
finite $W$-algebra is isomorphic to
a quotient of the super-Yangian of $Q({n\over l})$
\end{abstract}

\section{Introduction}
A finite $W$-algebra is a certain associative algebra attached to a
pair $(\g,e)$ where $\g$ is a complex semisimple Lie algebra and $e\in\g$
is a nilpotent element.
Geometrically a finite $W$ algebra is a
quantization of the Poisson structure on the so-called Slodowy slice
(a transversal slice to the orbit of $e$ in the adjoint
representation).

In the case when $e=0$ the finite $W$-algebra
coincides with the universal enveloping algebra $U(\g)$ and in the
case when $e$ is a {\it regular} nilpotent element, the corresponding
$W$-algebra coincides with the center of $U(\g)$. The latter case was
studied by B. Kostant \cite{Ko}
who was motivated by applications to generalized Toda lattices. The
general definition of a finite $W$-algebra was given by
A. Premet in \cite{Pr1} (see also \cite{L}).
In the case of Lie superalgebras,
finite $W$-algebras have been extensively studied by mathematicians and physicists in
[1, 2, 10-14, 18-20].

E. Ragoucy and P. Sorba first observed that
in the case when $\g$ is the general linear Lie algebra
and $e$ consists of $n$ Jordan blocks each of size $l$,
the finite $W$-algebra for $\g$ is isomorphic to the truncated Yangian
of level $l$ associated to $\g\l(n)$, which is a certain quotient of the Yangian $Y_{n}$ for $\g\l(n)$ \cite{RS}.
J. Brundan and A. Kleshchev generalized this result to an arbitrary nilpotent $e$, and obtained
a realization of the finite $W$-algebra for the general linear Lie algebra
as a quotient of a so-called shifted Yangian \cite{BK1} (see also \cite{BK2}).

For the general linear Lie superalgebra $\g = \g\l(m|n)$,
a connection between finite $W$-algebras for $\g$ and super-Yangians was firstly observed
by C. Briot and E. Ragoucy \cite{BR}.
In a more recent article, J. Brown, J. Brundan and S. Goodwin described
principal finite $W$-algebras for $\g\l(m|n)$ associated to regular (principal) nilpotent $e$ as truncations of shifted super-Yangians of $\g\l(1|1)$ \cite{BBG}.
After that, Y. Peng described the finite $W$-algebra for $\g = \g\l(m|n)$
 associated to an $e$ in the case when the Jordan type of $e$ satisfies the following condition: $e = e_m\oplus e_n$, where $e_m$ is principal nilpotent in $\g\l(m|0)$ and the sizes
of the Jordan blocks of $e_n$ are all greater or equal to $m$ \cite{Pe}.

For a Lie superalgebra $\g = \g_{\bar 0}\oplus \g_{\bar 1}$ with a
reductive even part $\g_{\bar 0}$ we denote
by $W_\chi$ the finite $W$-algebra associated
to an even  nilpotent element
$\chi\in \g_{\bar 0}^*\subset \g^*$ in the coadjoint representation.
If $\g$ admits an even non-degenerate invariant form, then $\g\cong\g^*$ and $\chi(x) = (e|x)$ for some nilpotent $e\in \g_{\bar 0}$. If $\g$ is the queer
 Lie superalgebra $Q(n)$, then it admits an odd non-degenerate invariant form. In this case $\g\cong\Pi(\g^*)$ and $\chi(x) = (E|x)$ for some nilpotent $E\in \g_{\bar 1}$.

In \cite{PS2} we studied in detail $Q(n)$ in the regular case. In particular, we proved that
$W_\chi$ for $Q(n)$ associated to a regular nilpotent $\chi$
is isomorphic to a quotient of the super-Yangian of $Q(1)$ (Theorem 6.2).
An interesting problem is to extend this result to $Q(n)$ associated to an {\it arbitrary} even nilpotent $\chi$.

In this work we consider  the case when the corresponding nilpotent element
has Jordan blocks each of size $l$.
We construct a set of generators  of $W_{\chi}$ (Theorem \ref {T2}) and prove that
$W_{\chi}$ is isomorphic to a quotient of the super-Yangian of $Q({n\over l})$
(Theorem \ref {T1} and Corollary \ref{corol}). This proves the conjecture, which we formulated in
\cite{P2}.

A. Premet has proved that if $\g$ is a semi-simple Lie algebra, then the associated
graded algebra $Gr_K W_{\chi}$ with respect to the Kazhdan filtration is isomorphic to $S(\g^{\chi})$
(the symmetric algebra of the annihilator $\g^{\chi}$ of $\chi$ in $\g$) (see \cite {Pr1}).

In \cite{PS2} we formulated the following conjecture (Conjecture 2.8):
 Assume that $\g$ is a Lie superalgebra with a reductive even part $\g_{\bar 0}$ endowed with a $\Z$-grading
 $\g = \oplus_{j\in\Z} \g_j$, which is good for $\chi$. Then
 if $\hbox{dim}(\g_{-1})_{\bar 1}$ is even, we have that $Gr_KW_\chi\simeq S(\g^\chi)$ and if
$\hbox{dim}(\g_{-1})_{\bar 1}$ is odd, then  $Gr_KW_\chi\simeq S(\g^\chi)\otimes \C[\xi]$,
where $\C[\xi]$ is the exterior algebra generated by one element $\xi$.

For $\g = \g\l(m|n)$ and a regular nilpotent $\chi$ this conjecture is proven in \cite{BBG}.
Recently Y. Zeng and B. Shu have proved this conjecture
for  a basic Lie superalgebra  $\g$ over $\C$ of any type except
$D(2,1;\alpha)$, where $\alpha\not\in\bar{\QQ}$ (\cite{ZS}, Theorem 0.1).
In \cite{P1} it is proven for $D(2,1;\alpha)$ and a regular nilpotent $\chi$.
We proved this conjecture for $\g = Q(n)$ in the regular case in \cite{PS2}
(Corollary 4.9). It follows from our results, that the conjecture is also true
in the case that we consider in this paper (Corollary \ref{cor}).

Notice that in Theorem \ref {T1} we realized the finite $W$-algebra for $Q(n)$ inside $U(Q({n\over l}))^{\otimes l}$ as
\begin{equation}
W_{\chi}\cong U^{\otimes l}\circ \Delta_l^{op} (Y(Q({n\over l}))),
\nonumber
\end{equation}

\noindent
where $\Delta^{op}$ is the opposite comultiplication in $Y(Q({n\over l}))$ and
$U$ is its homomorphism into $U(Q({n\over l}))$ defined in \cite{NS}.
Then we obtained a realization of the same subspace as
\begin{equation}
W_{\chi}\cong ev^{\otimes l}\circ \Delta_l (Y(Q({n\over l}))),
\nonumber
\end{equation}

\noindent
where $\Delta$ is the  comultiplication and $ev$ is the evaluation homomorphism (Corollary \ref {COR}).
Note that the latter realization is in the spirit of \cite {BK1} and \cite {BBG}, where the authors used (shifted) truncated (super)-Yangians.

\section{Finite $W$-algebras for Lie superalgebras}

Let $\g = \g_{\bar 0}\oplus \g_{\bar 1}$ be a Lie superalgebra with reductive even part $\g_{\bar 0}$.
Let $\chi\in \g_{\bar 0}^*\subset\g^*$ be an even nilpotent element
in the coadjoint representation, i.e. the closure of the $G_{\bar 0}$-orbit of $\chi$ in $\g_{\bar 0}^*$
(where $G_{\bar 0}$ is the algebraic reductive group of $\g_{\bar 0}$)
contains zero.

\begin{definition}
{\it The annihilator of $\chi$ in $\g$} is
$$\g^\chi=\{x\in\g \hbox{ }|\hbox{ }\chi([x,\g])=0\}.$$
\end{definition}

\begin{definition}
{\it A good $\Z$-grading for $\chi$} is a $\Z$-grading
$\g = \oplus_{j\in\Z} \g_j $ satisfying the following two conditions:

(1) $\chi(\g_j)=0$ if $j\neq -2$;

(2) $\g^\chi$  belongs to $\bigoplus_{j\geq 0} \g_j$.

\end{definition}

\noindent
Note that $\chi([\cdot,\cdot])$ defines a non-degenerate skew-symmetric even bilinear form on
$\g_{-1}$. Let $\l$ be a maximal isotropic subspace with respect to
this form. We consider
a nilpotent subalgebra
$\m = (\oplus_{j\leq -2}\g_j)\bigoplus \l$ of $\g$.
The restriction of $\chi$ to $\m$,
$\chi: \m\longrightarrow \C$,
defines a one-dimensional representation
$\C_{\chi} = <v>$ of $\m$.
Let $I_{\chi}$ be the left ideal of $U(\g)$
generated by $a - \chi(a)$ for all $a\in \m$.

\begin{definition}
The induced $\g$-module
$$Q_{\chi} := U(\g)\otimes_{U(\m)}\C_{\chi} \cong U(\g)/I_{\chi}$$
is called {\it the generalized Whittaker module}.
\end{definition}

\begin{definition}
{\it The finite $W$-algebra} associated to the nilpotent element
$\chi$ is
$$W_{\chi} := \hbox{End}_{U(\g)}(Q_{\chi})^{op}.$$
\end{definition}
As in the Lie algebra case, the superalgebras $W_\chi$ are all
isomorphic for different choices of good $\Z$-gradings and maximal
isotropic subspaces $\l$ \cite{Z}.

\noindent
If $\g$ admits an even non-degenerate $\g$-invariant supersymmetric bilinear form, then $\g\simeq \g^*$
and $\chi(x)=(e|x)$ for some nilpotent $e\in \g_{\bar 0}$ (i.e. $\hbox{ad}e$ is a nilpotent endomorphism of $\g$).
By the Jacobson--Morozov theorem  $e$ can be included in $\s\l(2) = <e, h, f>$.
As in the Lie algebra case, the linear operator $\hbox{ad} h$ defines a Dynkin $\Z$-grading
$\g = \bigoplus_{j\in \Z}\g_j$, where
$$\g_j = \lbrace x\in \g \hbox{ }|\hbox{ }\hbox{ad} h(x) = jx\rbrace.$$
As follows from the representation theory of $\s\l(2)$, the Dynkin $\Z$-grading
is good for $\chi$.
Let $\g^e := \hbox{Ker}(\hbox{ad} e)$. Clearly, $\g^e = \g^{\chi}$.
Note that as in the Lie algebra case,
$\dim \g^e = \dim \g_0 + \dim \g_1$.

Note that by Frobenius reciprocity
$$\hbox{End}_{U(\g)}(Q_{\chi})= \hbox{Hom}_{U(\m)}(\C_{\chi},Q_{\chi}).$$
That defines an identification of $W_\chi$ with the subspace
$$Q_\chi^{\m}= \lbrace u\in Q_{\chi}\hbox{ }| \hbox{ }au=\chi(a)u\hbox{ for all } a\in \m\rbrace.$$
In what follows we denote by  $\pi:U(\g)\to U(\g)/I_\chi$  the natural
projection. By above
$$W_{\chi} = \lbrace \pi(y) \in U(\g)/I_{\chi} \hbox{ }| \hbox{ } (a - \chi(a))y\in I_{\chi}\hbox{ for all } a\in \m\rbrace,$$
or, equivalently,
\begin{equation}\label{walg}
W_{\chi} = \lbrace \pi(y) \in U(\g)/I_{\chi} \hbox{ }| \hbox{ ad} (a)y\in I_{\chi}\hbox{ for all } a\in \m\rbrace.
\end{equation}
The algebra structure on $W_{\chi}$ is given by
$$\pi(y_1)\pi({y}_2) = \pi({y_1y_2})$$
\vskip 0.1in
\noindent
for $y_i\in U(\g)$ such that $\hbox{ ad}(a)y_i\in I_{\chi} \hbox{ for all }  a\in \m$
and $i = 1, 2$.

\begin{definition}
A nilpotent $\chi\in\g^*_{\bar 0}$ is called {\it regular} nilpotent if $G_{\bar 0}$-orbit of $\chi$ has
maximal dimension, i.e. the dimension of $\g^\chi_{\bar 0}$ is minimal.
(Equivalently, a nilpotent $e\in\g_{\bar 0}$ is  {\it regular} nilpotent, if
$\g^e_{\bar 0}$ attains the minimal dimension, which is equal to $\hbox{rank}\g_{\bar 0}$.)
\end{definition}

\begin{theorem}(B. Kostant, \cite{Ko})
For a reductive Lie algebra $\g$ and
a regular nilpotent element $e\in\g$, the finite $W$-algebra
$W_{\chi}$ is isomorphic to the center  of $U(\g)$.
\end{theorem}
This theorem does not hold for Lie superalgebras, since $W_{\chi}$ must have a non-trivial odd part,
and the center  of $U(\g)$ is even.

\begin{definition}
Define a $\Z$-grading
on $T(\g)$ by setting the degree of $g\in \g_j$ to be $j+2$. This induces a
filtration on $U(\g)$ and therefore on $U(\g)/I_\chi$ which is called
the {\it Kazhdan filtration}. We
will denote by $Gr_K$ the corresponding graded algebras. Since by (\ref{walg})
$W_\chi\subset U(\g)/I_\chi$, we have an induced filtration on $W_\chi$.
\end{definition}

\begin{theorem}(A. Premet, \cite{Pr1})
Let $\g$ be a semi-simple Lie algebra.
Then the associated graded algebra $Gr_KW_{\chi}$ is isomorphic to $S(\g^{\chi})$.
\end{theorem}

\noindent
To generalize this result to the super case, we assume that
$\l'$ is some subspace in $\g_{-1}$ satisfying the following two properties:

(1) $\g_{-1}=\l\oplus\l'$,

(2) $\l'$ contains a maximal isotropic subspace with respect to the form defined by $\chi([\cdot,\cdot])$ on $\g_{-1}$.

If $\hbox{dim}(\g_{-1})_{\bar 1}$ is even, then $\l'$ is a maximal isotropic subspace. If $\hbox{dim}(\g_{-1})_{\bar 1}$ is
odd, then $\l^\perp\cap \l'$ is one-dimensional and we fix $\theta\in \l^\perp\cap \l'$ such that $\chi([\theta,\theta])=2$. It is clear that $\pi(\theta)\in W_\chi$
and $\pi(\theta)^2=1$.

Let $\p=\bigoplus _{j\geq 0}\g_j$.
By the PBW theorem, $U(\g)/I_\chi\simeq S(\p\oplus\l')$ as a vector space. Therefore
${Gr}_K(U(\g)/I_\chi)$ is isomorphic to $S(\p\oplus\l')$ as a vector space.
The good $\Z$-grading of $\g$ induces the grading on $S(\p\oplus\l')$.
For any $X\in U(\g)/I_\chi$ let $Gr_K(X)$ denote the corresponding element in ${Gr}_K(U(\g)/I_\chi)$,
and $P(X)$ denote the highest weight component of $Gr_K(X)$ in this $\Z$-grading.

\begin{theorem}\label{maint}
(\cite{PS2}, Proposition 2.7)
  Let $y_1,\dots,y_p$ be a basis in $\g^{\chi}$ homogeneous in the good $\Z$-grading.
Assume that there exist $Y_1,\dots, Y_p\in W_\chi$ such that $P(Y_i)=y_i$ for all $i=1,\dots, p$. Then

\noindent
(a) if $\hbox{dim}(\g_{-1})_{\bar 1}$ is even, then  $Y_1,\dots,Y_p$ generate $W_\chi$,  and
if $\hbox{dim}(\g_{-1})_{\bar 1}$ is odd, then
$Y_1,\dots,Y_p$ and $\pi(\theta)$ generate $W_\chi$;

\noindent
(b) if $\hbox{dim}(\g_{-1})_{\bar 1}$ is even, then ${Gr}_KW_\chi\simeq S(\g^{\chi})$, and
 if $\hbox{dim}(\g_{-1})_{\bar 1}$ is odd, then
 $Gr_KW_\chi\simeq S(\g^{\chi})\otimes \C[\xi]$, where $\C[\xi]$ is the exterior algebra generated by one element $\xi$.
\end{theorem}

 \section{The queer superalgebra $Q(n)$}

  Recall that
the {\it queer} Lie superalgebra is defined as follows
$$Q(n) := \lbrace
\left(\begin{array}{c|c}
A&B\\
\hline
B&A\\
\end{array}\right)\hbox{ }|\hbox{ }A, B \hbox{ are } n\times n \hbox{ matrices}\rbrace.$$
Let $\hbox{otr} \left(\begin{array}{c|c}
A&B\\
\hline
B&A\\
\end{array}\right) = \hbox{tr} B$.

\noindent
\begin{remark}
$Q(n)$ has one-dimensional center $<z>$, where $z = 1_{2n}$.
Let
$$S{Q}(n) = \lbrace X\in Q(n)
\hbox{ }|\hbox{ otr}X = 0\rbrace.$$

\noindent
The Lie superalgebra $\tilde{Q}(n) := S{Q}(n)/<z>$ is simple for $n\geq 3$, see \cite{K}.

\end{remark}

Note that
$\g = Q(n)$ admits an {\it odd} non-degenerate $\g$-invariant supersymmetric bilinear form
$$(x | y) := \hbox{otr}(xy) \hbox{ for } x, y\in \g.$$
Therefore,
we identify the coadjoint module $\g^*$ with $\Pi(\g)$, where $\Pi$ is
the  change of parity functor.

\noindent
Let $e_{i,j}$ and $f_{i,j}$ be standard bases in $\g_{\bar 0}$ and $\g_{\bar 1}$ respectively:

\vskip 0.2in
$$e_{i,j} =
\left(\begin{array}{c|c}
E_{ij}&0\\
\hline
0&E_{ij}\\
\end{array}\right), \quad
f_{i,j} =
\left(\begin{array}{c|c}
0&E_{ij}\\
\hline
E_{ij}&0\\
\end{array}\right),$$
where $E_{ij}$ are elementary $n\times n$  matrices.

Let $\s\l(2) = <e, h, f>$, where
\begin{align}
&e = \sum_{p = 1}^{n\over l}\sum_{i=1}^{l-1}
e_{l(p-1)+i,l(p-1)+i+1},\nonumber\\
&f = \sum_{p = 1}^{n\over l}\sum_{i=1}^{l-1} i(l-i)
e_{l(p-1)+i+1,l(p-1)+i},\nonumber\\
&h = \sum_{p = 1}^{n\over l}\sum_{i=1}^{l}(l-2i+1)
e_{l(p-1)+i,l(p-1)+i}.\nonumber
\end{align}
Note that $e$ is an even nilpotent element in $Q(n)$, which is represented by a nilpotent $n\times n$-matrix
whose Jordan blocks are all of the same size $l$. Note also that
$\hbox{ad}h$ defines an even Dynkin $\Z$-grading of
$\g$:
$$\g = \bigoplus_{k = 1-l}^{l-1}\g_{2k},$$
where
\begin{align}
&\g_{2k} = <e_{l(p-1)+i,l(q-1)+i+k}\hbox{ }|\hbox{ }f_{l(p-1)+i,l(q-1)+i+k}>,
\nonumber\\
&\g_{-2k} = <e_{l(p-1)+i+k,l(q-1)+i}\hbox{ }|\hbox{ }f_{l(p-1)+i+k,l(q-1)+i}>,
\nonumber\\
&\hbox{where } k = 0, 1, \ldots, l-1, i = 1, \ldots, l-k \hbox{ and } 1\leq p,q \leq  {n\over l}.
\nonumber
\end{align}
Let  $E = \sum_{p = 1}^{n\over l}\sum_{i=1}^{l-1}
f_{l(p-1)+i,l(p-1)+i+1}$.
Since we have an isomorphism $\g^*\simeq\Pi(\g)$, an even  nilpotent $\chi\in\g^*$ can be defined by
$\chi(x) := (x|E)$ for $x \in \g$.
Note that the Dynkin $\Z$-grading is good for $\chi$.
We have that

\begin{align}
&\g^\chi=\g^E =
<\sum_{i=1}^{l-k}e_{l(p-1)+i,l(q-1)+i+k} \hbox{ }|\hbox{ }
\sum_{i=1}^{l-k}(-1)^{i+k-1}f_{l(p-1)+i,l(q-1)+i+k}>,
\nonumber\\
&\hbox{where } k = 0, 1, \ldots, l-1, \hbox{ and } 1\leq p,q \leq  {n\over l}.
\nonumber
\end{align}
Thus $\dim(\g^E) = ({n^2\over l}|{n^2\over l})$.
Note that  as in the Lie algebra case,
$\dim (\g^E) = \dim \g_0$, since the $\Z$-grading is even,
and $\g^E \subseteq \bigoplus_{k = 0}^{l-1}\g_{2k}$.
Note also that $\chi$ is regular nilpotent if and only if $e$ has a single Jordan block,
i.e. $l = n$.

Let
$$\m = \bigoplus_{j = 1}^{l-1}\g_{-2j}.$$
Note that $\m$ is generated by $e_{l(p-1)+i+1, l(q-1)+i}$ and
$f_{l(p-1)+i+1, l(q-1)+ i}$, where $i = 1, \ldots, l-1$, and
$1\leq p,q \leq  {n\over l}$. We have that
\begin{align}\label{equA1}
&\chi(e_{l(p-1)+i+1, l(q-1)+i}) = \delta_{p,q}, \hbox{ for  } i = 1, \ldots, l-1,\nonumber\\
&\chi(e_{l(p-1)+i+k, l(q-1)+i}) = 0 \hbox{ for  } k\geq 2, i = 1, \ldots, l-k,\\
&\chi(f_{l(p-1)+i+k, l(q-1)+i}) = 0
\hbox{ for  } k\geq 1, i = 1, \ldots, l-k.\nonumber
\end{align}
The left ideal $I_\chi$ and $W_\chi$ are defined now as usual. Moreover,
$$\p:=\bigoplus_{j = 0}^{l-1}\g_{2j}$$
is a parabolic subalgebra of $\g$ and $\p=\g_0\oplus\n$, where
$$\n:=\bigoplus_{j = 1}^{l-1}\g_{2j}.$$
Note that since the $\Z$-grading of $\g$ is even, then
the algebra $W_\chi$ can be regarded as a {\it subalgebra} of $U(\p)$.
In fact, from the PBW theorem,
$$U(\g) = U(\p) \oplus I_{\chi}.$$
The projection $pr: U(\g) \longrightarrow U(\p)$
along this direct sum decomposition induces an isomorphism:
$U(\g)/I_{\chi} \buildrel \sim \over \longrightarrow U(\p)$.
Let
$$U(\p)^+:=\oplus_{i>0} U(\p)_{2i}.$$
It is a two sided ideal in $U(\p)$ and $U(\p)/U(\p)^+ \cong U(\g_0)$.
Let $\vartheta: U(\p)\longrightarrow U(\g_0)$ be the natural projection.
Its restriction to $W_\chi$ is the {\it Harish-Chandra homomorphism}
$$\vartheta: W_\chi\longrightarrow U(\g_0),$$
which is injective by \cite{PS2} (Theorem 3.1).

 \section{Super-Yangian of $Q(n)$}

 Super-Yangian $Y(Q(n))$  was introduced  by M. Nazarov in \cite{N}.
Recall that
$Y(Q(n))$ is the associative unital superalgebra over $\C$ with the countable set of generators
$$T_{i,j}^{(m)}\hbox{ where }m = 1, 2, \ldots \hbox{ and }i, j = \pm 1, \pm 2, \ldots, \pm n.$$

\noindent
The $\Z_2$-grading of the algebra $Y(Q(n))$ is defined as follows:
$$p(T_{i,j}^{(m)})  =
p(i) + p(j), \hbox{ where } p(i) = 0 \hbox{ if } i>0, \hbox{ and } p(i) = 1  \hbox{ if }i<0.$$
To write down defining relations for these generators we employ the formal series

\noindent
in $Y(Q(n))[[u^{-1}]]$:
\begin{equation}\label{equY1}
T_{i,j}(u) = \delta_{i,j}\cdot 1 + T_{i,j}^{(1)}u^{-1} + T_{i,j}^{(2)}u^{-2} + \ldots.
\end{equation}
Then for all possible indices $i, j, k, l$ we have the relations

\begin{align}\label{equY2}
& (u^2 - v^2)[T_{i,j}(u), T_{k,l}(v)]\cdot (-1)^{p(i)p(k) + p(i)p(l)  +  p(k)p(l)}\\
& = (u + v)(T_{k,j}(u)T_{i,l}(v) - T_{k,j}(v)T_{i,l}(u))\nonumber\\
&-(u - v)(T_{-k,j}(u)T_{-i,l}(v) - T_{k,-j}(v)T_{i,-l}(u))\cdot (-1)^{p(k) +p(l)},
\nonumber
\end{align}
 where $v$ is a formal parameter independent of $u$, so that (\ref{equY2}) is an equality in the algebra of formal Laurent series in $u^{-1}, v^{-1}$ with
coefficients in $Y(Q(n))$.

\noindent
 For all indices $i, j$ we also have the relations
 \begin{equation}\label{equY3}
 T_{i,j}(-u) = T_{-i,-j}(u).
 \end{equation}

\noindent
Note that the relations (\ref{equY2}) and (\ref{equY3}) are equivalent to the following defining relations:

 \begin{align}\label{equY4}
 & ([T_{i,j}^{(m+1)}, T_{k,l}^{(r-1)}] - [T_{i,j}^{(m-1)}, T_{k,l}^{(r+1)}])
 \cdot (-1)^{p(i)p(k) + p(i)p(l)  + p(k)p(l)}  = \\
 & T_{k,j}^{(m)}T_{i,l}^{(r-1)} + T_{k,j}^{(m-1)}T_{i,l}^{(r)} -
 T_{k,j}^{(r-1)}T_{i,l}^{(m)} - T_{k,j}^{(r)}T_{i,l}^{(m-1)}\nonumber\\
 & + (-1)^{p(k) + p(l)}(-T_{-k,j}^{(m)}T_{-i,l}^{(r-1)} + T_{-k,j}^{(m-1)}T_{-i,l}^{(r)} +
 T_{k,-j}^{(r-1)}T_{i,-l}^{(m)} - T_{k,-j}^{(r)}T_{i,-l}^{(m-1)}),
 \nonumber
\end{align}

\begin{equation}\label{equY5}
 T_{-i,-j}^{(m)} = (-1)^m T_{i,j}^{(m)},
\end{equation}
where $m, r = 1, \ldots$ and $T_{i,j}^{(0)} = \delta_{i,j}$.

Recall that $Y(Q(n))$ is a Hopf superalgebra   with comultiplication given by the formula
\begin{equation}\label{delta}
\Delta(T_{i,j}^{(r)})=\sum_{s=0}^r\sum_{k} (-1)^{(p(i)+p(k))(p(j)+p(k))}T_{i,k}^{(s)}\otimes T_{k,j}^{(r-s)}.
\end{equation}
The opposite comultiplication is given by
\begin{equation}\label{opp}
\Delta^{op}(T_{i,j}^{(r)})=\sum_{s=0}^r\sum_{k}  T_{k,j}^{(r-s)}\otimes T_{i,k}^{(s)}.
\end{equation}
Combine the series (\ref{equY1}) into the single element
$$T(u) = \sum_{i,j}E_{i,j}\otimes T_{i,j}(u)$$
of the algebra $\hbox{End}(\C^{n|n})\otimes Y(Q(n))[[u^{-1}]]$.
The element $T(u)$ is invertible and we put
\begin{equation}\label{tilda}
T(u)^{-1} = \sum_{i,j}E_{i,j}\otimes \tilde{T}_{i,j}(u).
\end{equation}
The assignment
$T_{i,j}(u) \mapsto \tilde{T}_{i,j}(u)$ defines the {\it antipodal map}
\begin{equation}\label{antipodal}
S: Y(Q(n)) \longrightarrow Y(Q(n)),
\end{equation}
which is an anti-automorphism of the $\Z_2$-graded algebra
$Y(Q(n))$. Recall that an {\it anti-homomorphism} $\varphi: \A\rightarrow \B$ of associative Lie superalgebras
is a linear map, which preserves the $\Z_2$-grading and satisfies for any homogeneous $X, Y\in \A$
$$\varphi(XY) = (-1)^{p(X)p(Y)}\varphi(Y)\varphi(X).$$
Note that
$T(u) = I + A$, \hbox{where }
$$I = \sum_{i,j}E_{i,j}\otimes\delta_{i,j}, \quad
A = \sum_{i,j}\sum_{r>0}E_{i,j}\otimes T_{i,j}^{(r)}u^{-r}.$$
Then
$$T(u)^{-1} = I - A + A^2 - A^3 + \ldots$$
Note that
$$(E_{i,j}\otimes T_{i,j}^{(r_1)}u^{-r_1})
(E_{j,k}\otimes T_{j,k}^{(r_2)}u^{-r_2}) =
(-1)^{(p(i) + p(j))(p(j) + p(k))}(E_{i,k}\otimes T_{i,j}^{(r_1)}T_{j,k}^{(r_2)}u^{-r_1-r_2}).$$
Hence, if $r\geq 1$, then
\begin{align}\label{antiS}
&S(T_{i,j}^{(r)}) = -T_{i,j}^{(r)} + \\
&\sum_{m=2}^{\infty} (-1)^m
(\sum_{i_1, \ldots i_{m-1}\in {\lbrace \pm 1, \ldots \pm n\rbrace}}
(\sum_{\substack{r_1 + \ldots + r_m = r\\r_1, \ldots, r_m >0}}
(-1)^{\nu (i,i_1,\ldots, i_{m-1},j)}
T_{i,i_1}^{(r_1)}T_{i_1,i_2}^{(r_2)}\ldots T_{i_{m-1},j}^{(r_m)})),\nonumber
\end{align}
where
$$\nu (i_0,i_1,\ldots, i_{m-1},i_m) =
\sum_{0\leq k<s\leq m-1}(p(i_k) + p(i_{k+1})(p(i_s)+p(i_{s+1})).$$
If $r=0$, then $S(T_{i,j}^{(0)}) = \delta_{i,j}$.

The evaluation homomorphism ${ev}: Y(Q(n))\to U(Q(n))$ is defined as follows
\begin{equation}\label{eval}
T_{i,j}^{(1)}\mapsto -e_{j,i},\quad T_{-i,j}^{(1)}\mapsto -f_{j,i}  \hbox{ for } i, j > 0, \quad T_{i,j}^{(0)}\mapsto \delta_{i,j}, \quad T_{i,j}^{(r)}\mapsto 0 \hbox{ for } r > 1.
\end{equation}
Recall that for any Lie superalgebra $\g$, the principal anti-automorphism $\alpha$ of the enveloping superalgebra
$U(\g)$ is defined by the assignment $\alpha: X\mapsto -X$ for all $X\in \g$.
Then ${\bar{ev}}:= \alpha\circ ev$ in an anti-homomorphism ${\bar{ev}}: Y(Q(n))\to U(Q(n))$, and it is given as follows
\begin{equation}\label{antieval}
T_{i,j}^{(1)}\mapsto e_{j,i},\quad T_{-i,j}^{(1)}\mapsto f_{j,i}  \hbox{ for } i, j > 0, \quad T_{i,j}^{(0)}\mapsto \delta_{i,j}, \quad T_{i,j}^{(r)}\mapsto 0 \hbox{ for } r > 1.
\end{equation}
In \cite{S} A. Sergeev recursively defined
the elements $e_{i,j}^{(m)}$ and $f_{i,j}^{(m)}$ of  $U(Q(n))$:
\begin{align}\label{equG1}
& e_{i,j}^{(m)} = \sum_{k = 1}^n e_{i,k}e_{k,j}^{(m-1)} +
(-1)^{m+1}\sum_{k = 1}^n f_{i,k}f_{k,j}^{(m-1)},\\
&f_{i,j}^{(m)} = \sum_{k = 1}^n e_{i,k}f_{k,j}^{(m-1)} +
(-1)^{m+1}\sum_{k = 1}^n f_{i,k}e_{k,j}^{(m-1)},
\nonumber
\end{align}
where $e_{i,j}^{(0)} = \delta_{i,j}$ and $f_{i,j}^{(0)} = 0$.
Then
\begin{align}\label{equG2}
&[e_{i,j}, e_{k,l}^{(m)}] = \delta_{j,k}e_{i,l}^{(m)} - \delta_{i,l}e_{k,j}^{(m)},\quad
[e_{i,j}, f_{k,l}^{(m)}] = \delta_{j,k}f_{i,l}^{(m)} - \delta_{i,l}f_{kj}^{(m)},\\
&[f_{i,j}, e_{k,l}^{(m)}] = (-1)^{m+1}\delta_{j,k}f_{i,l}^{(m)} - \delta_{i,l}f_{k,j}^{(m)},\nonumber\\
&[f_{i,j}, f_{k,l}^{(m)}] = (-1)^{m+1}\delta_{j,k}e_{i,l}^{(m)} + \delta_{i,l}e_{k,j}^{(m)}.
\nonumber
\end{align}

\begin{lemma}\label{LemU}
There exists  a homomorphism
$U: Y(Q(n))\to U(Q(n))$ defined as follows
\begin{align}\label{homomorphism}
&T_{i,j}^{(r)}\mapsto (-1)^r e_{j,i}^{(r)}, \hbox{ if } i>0, j>0, r>0,\nonumber\\
&T_{i,j}^{(r)}\mapsto (-1)^r f_{j,-i}^{(r)}, \hbox{ if } i<0, j>0, r>0,
T_{i,j}^{(0)}\mapsto \delta_{i, j}.
\end{align}
\end{lemma}

\begin{proof}
It follows from \cite{NS} (Proposition 1.6) that one can define an anti-homomorphism
$\omega: Y(Q(n))\to U(Q(n))$ by
\begin{align}\label{homomorphism}
&T_{i,j}^{(r)}\mapsto e_{j,i}^{(r)}, \hbox{ if } i>0, j>0, r>0,\nonumber\\
&T_{i,j}^{(r)}\mapsto f_{j,-i}^{(r)}, \hbox{ if } i<0, j>0, r>0,
T_{i,j}^{(0)}\mapsto \delta_{i, j}.
\end{align}
On the other hand, there exists an  anti-automorphism $\beta$ of $Y(Q(n))$
defined as follows:
\begin{equation}\label{antiQ}
\beta: T_{i,j}^{(r)}\longrightarrow (-1)^r T_{i,j}^{(r)}.
\end{equation}
One can easily verify that $\beta$ preserves the relation (\ref{equY2}).
In fact, according to (\ref{equY3}) and (\ref{equY5})
$\beta(T_{i,j}(u)) = T_{i,j}(-u)$. One can apply $\beta$ to equation  (\ref{equY2})
and identify: $u\leftrightarrow -v, v\leftrightarrow -u, i\leftrightarrow k, j\leftrightarrow l$.
Hence $U := \omega\circ \beta$ is a homomorphism.
\end{proof}
Finally, observe that the maps

\begin{align}
&\Delta_{l}: Y(Q(n))\longrightarrow Y(Q(n))^{\otimes l},\nonumber\\
&\Delta_{l}^{op}: Y(Q(n))\longrightarrow Y(Q(n))^{\otimes l},
\nonumber
\end{align}
where
\begin{align}
&\Delta_l := \Delta_{l-1, l}\circ \cdots\circ\Delta_{2, 3}\circ\Delta,\nonumber\\
&\Delta_l^{op} := \Delta_{l-1, l}^{op}\circ \cdots\circ\Delta_{2, 3}^{op}\circ\Delta^{op}
\nonumber
\end{align}
are homomorphisms of associative algebras.

\noindent
We will  prove the following statement.

\noindent
\begin{theorem}\label{T1}
Let $e$ be an even nilpotent element in $Q(n)$ whose Jordan blocks are each of size $l$.
Then the finite $W$-algebra for $Q(n)$
is isomorphic to the image of $Y(Q({n\over l}))$ under the homomorphism
$$U^{\otimes l}\circ \Delta_l^{op}: Y(Q({n\over l}))\longrightarrow
(U(Q({n\over l})))^{\otimes l}.$$
\end{theorem}

\begin{remark}\label{Rem1}
In \cite{PS2} we proved this theorem in the case when $e$ is a regular even nilpotent element, i.e.
$l = n$.
\end{remark}

\section{Generators of $W_\chi$ for the queer Lie superalgebra $Q(n)$}

In this section we construct some generators of $W_\chi$. In particular, we will prove that $W_\chi$ is finitely generated.
We use the elements $e_{i,j}^{(m)}$ and $f_{i,j}^{(m)}$ defined in (\ref{equG1}).

\noindent
\begin{theorem}\label{T2}
$\pi(e_{lp,l(q-1)+1}^{(l+k-1)})$ and $\pi(f_{lp,l(q-1)+1}^{(l+k-1)})$
 for $p, q = 1, \ldots, {n\over l}$ and $k = 1, \ldots, l$ generate
$W_\chi$.
\end{theorem}

\noindent
\begin{lemma}\label{Lem1}
Let $1\leq r\leq l-1$. Then
\begin{align}\label{equG3}
&\pi(e_{l(p-1)+m, l(q-1)+1}^{(r)})=\left\{ \begin{array}{cc}
 \delta_{p,q} &   if \quad  m = r+1,\\
 0 &  if \quad  r+2\leq m \leq l,\\
\end{array}\right.\\
&\pi(f_{l(p-1)+m,l(q-1)+1}^{(r)})=0, \hbox{ if } r+1\leq m \leq l.
\nonumber
\end{align}
\end{lemma}

\begin{proof}
We will prove the statement by induction on $r$. For $r = 1$ we have that
\begin{equation}
\pi(e_{l(p-1)+m,l(q-1) + 1}^{(1)})=
\pi(e_{l(p-1)+m,l(q-1)+1}),\quad
\pi(f_{l(p-1)+m, l(q-1)+1}^{(1)})=\pi(f_{l(p-1)+m, l(q-1)+1}).
\nonumber
\end{equation}
Then (\ref{equG3}) follows from  (\ref{equA1}).
Assume that (\ref{equG3}) holds for $r$.
From (\ref{equG1}) we have that
\begin{align}
&e_{l(p-1)+m, l(q-1)+1}^{(r+1)} =
\sum_{s =1}^{n\over l}\Big[
\sum_{k = 1}^l e_{l(p-1)+m,l(s-1)+k}e_{l(s-1)+k, l(q-1)+1}^{(r)} +
\nonumber\\
&(-1)^{r}\sum_{k = 1}^l f_{l(p-1)+m,l(s-1)+k}f_{l(s-1)+k, l(q-1)+1}^{(r)}\Big],
\nonumber\\
&f_{l(p-1)+m, l(q-1)+1}^{(r+1)} =
\sum_{s =1}^{n\over l}\Big[
\sum_{k = 1}^l e_{l(p-1)+m,l(s-1)+k}f_{l(s-1)+k,l(q-1)+1}^{(r)} +
\nonumber\\
&(-1)^{r}\sum_{k = 1}^l f_{l(p-1)+m, l(s-1)+k}e_{l(s-1)+k,l(q-1)+1}^{(r)}\Big].
\nonumber
\end{align}
Note that from (\ref{equG2})
\begin{align}
&[e_{l(p-1)+m,l(s-1)+k}, e_{l(s-1)+k,l(q-1)+1}^{(r)}] =  e_{l(p-1)+m,l(q-1)+1}^{(r)},\nonumber\\
&[e_{l(p-1)+m,l(s-1)+k}, f_{l(s-1)+k,l(q-1)+1}^{(r)}] =  f_{l(p-1)+m,l(q-1)+1}^{(r)},
\nonumber
\end{align}
\begin{align}
&[f_{l(p-1)+m,l(s-1)+k}, e_{l(s-1)+k,l(q-1)+1}^{(r)}] =
(-1)^{r+1}f_{l(p-1)+m,l(q-1)+1}^{(r)},\nonumber\\
&[f_{l(p-1)+m, l(s-1)+k}, f_{l(s-1)+k,l(q-1)+1}^{(r)}] =
(-1)^{r+1}e_{l(p-1)+m, l(q-1)+ 1}^{(r)}.
\nonumber
\end{align}
Hence
\begin{align}
& e_{l(p-1)+m,l(q-1)+1}^{(r+1)} =
\sum_{s =1}^{n\over l}\Big[
\sum_{k = 1}^{m-1}(e_{l(s-1)+k,l(q-1)+1}^{(r)}e_{l(p-1)+m,l(s-1)+k} +
e_{l(p-1)+m,l(q-1)+ 1}^{(r)}) + \nonumber\\
&\sum_{k = m}^{l}e_{l(p-1)+m,l(s-1)+k}e_{l(s-1)+k,l(q-1)+1}^{(r)} +
(-1)^{r}\Big(\sum_{k = 1}^{m-1}(-f_{l(s-1)+k,l(q-1)+1}^{(r)}f_{l(p-1)+m,l(s-1)+k} + \nonumber\\
&(-1)^{r+1}e_{l(p-1)+m,l(q-1)+1}^{(r)}) +
\sum_{k=m}^{l}f_{l(p-1)+m,l(s-1)+k}f_{l(s-1)+k,l(q-1)+1}^{(r)}\Big)\Big],
\nonumber
\end{align}

\begin{align}
&f_{l(p-1)+m,l(q-1)+1}^{(r+1)} =
\sum_{s =1}^{n\over l}\Big[
\sum_{k = 1}^{m-1}(f_{l(s-1)+k,l(q-1)+1}^{(r)}e_{l(p-1)+m,l(s-1)+k} +
f_{l(p-1)+m,l(q-1)+1}^{(r)}) +\nonumber\\
&\sum_{k = m}^{l}e_{l(p-1)+m,l(s-1)+k}f_{l(s-1)+k,l(q-1)+1}^{(r)} +
(-1)^{r}\Big(\sum_{k = 1}^{m-1}(e_{l(s-1)+k,l(q-1)+1}^{(r)}f_{l(p-1)+m,l(s-1)+k} + \nonumber\\
&(-1)^{r+1}f_{l(p-1)+m, l(q-1)+1}^{(r)}) +
\sum_{k = m}^{l}f_{l(p-1)+m,l(s-1)+k}e_{l(s-1)+k,l(q-1)+1}^{(r)}\Big)\Big].
\nonumber
\end{align}

\noindent
Then
\begin{align}
&\pi(e_{l(p-1)+m,l(q-1)+1}^{(r+1)}) =
\sum_{s =1}^{n\over l}\Big[
\sum_{k = 1}^{m-1}\pi(e_{l(s-1)+k,l(q-1)+1}^{(r)})\pi(e_{l(p-1)+m,l(s-1)+k})  + \nonumber\\
&\sum_{k = m}^{l}\Big(\pi(e_{l(p-1)+m,l(s-1)+k})\pi(e_{l(s-1)+k,l(q-1)+1}^{(r)}) +
(-1)^{r}\pi(f_{l(p-1)+m,l(s-1)+ k})\pi(f_{l(s-1)+ k,l(q-1)+1}^{(r)})\Big) + \nonumber\\
&(-1)^{r+1}\Big(\sum_{k = 1}^{m-1}\pi(f_{l(s-1)+ k,l(q-1)+1}^{(r)})
\pi(f_{l(p-1)+m,l(s-1)+k})\Big)\Big], \nonumber
\end{align}
\begin{align}
&\pi(f_{l(p-1)+m,l(q-1)+1}^{(r+1)}) =
\sum_{s =1}^{n\over l}\Big[
\sum_{k = 1}^{m-1}\pi(f_{l(s-1)+k,l(q-1)+1}^{(r)})\pi(e_{l(p-1)+m,l(s-1)+ k})  +\nonumber\\
&\sum_{k = m}^{l}\Big(\pi(e_{l(p-1)+m,l(s-1)+ k})\pi(f_{l(s-1)+ k,l(q-1)+1}^{(r)}) + (-1)^{r}\pi(f_{l(p-1)+m,l(s-1)+ k})\pi(e_{l(s-1)+ k,l(q-1)+1}^{(r)})\Big)+ \nonumber\\
&(-1)^{r}\Big(\sum_{k = 1}^{m-1}\pi(e_{l(s-1)+k,l(q-1)+1}^{(r)})\pi(f_{l(p-1)+m,l(s-1)+k})\Big)\Big].
\nonumber
\end{align}

\noindent
Then by (\ref{equA1})
\begin{align}
&\pi(e_{l(p-1)+m,l(q-1)+1}^{(r+1)}) = \pi(e_{l(p-1)+m-1,l(q-1)+1}^{(r)}) +
\sum_{s =1}^{n\over l}\Big[
\sum_{k = m}^{l}\Big(\pi(e_{l(p-1)+m,l(s-1)+k})
\pi(e_{l(s-1)+k,l(q-1)+1}^{(r)}) + \nonumber\\
&(-1)^{r}\pi(f_{l(p-1)+m,l(s-1)+k})\pi(f_{l(s-1)+k,l(q-1)+1}^{(r)})\Big)\Big],
\nonumber\\
&\pi(f_{l(p-1)+m,l(q-1)+1}^{(r+1)}) =
\pi(f_{l(p-1)+m-1,l(q-1)+1}^{(r)}) +
\sum_{s =1}^{n\over l}\Big[
\sum_{k = m}^{l}\Big(\pi(e_{l(p-1)+m,l(s-1)+k})
\pi(f_{l(s-1)+k,l(q-1)+1}^{(r)}) + \nonumber\\
&(-1)^{r}\pi(f_{l(p-1)+m,l(s-1)+k})\pi(e_{l(s-1)+k,l(q-1)+1}^{(r)})\Big)\Big].
\nonumber
\end{align}
Let $m\geq r+2$. Then by induction hypothesis,
\begin{equation}
\pi(e_{l(s-1)+k,l(q-1)+1}^{(r)}) = \pi(f_{l(s-1)+k,l(q-1)+1}^{(r)}) = 0
\hbox{ for }k = m, \ldots, l.
\nonumber
\end{equation}
If $m = r+2$, then
$\pi(e_{l(p-1)+m,l(q-1)+1}^{(r+1)}) =
\pi(e_{l(p-1)+r+1,l(q-1)+1}^{(r)}) = \delta_{p,q}$.
If $m \geq  r+3$, then $\pi(e_{l(p-1)+m,l(q-1)+1}^{(r+1)}) = \pi(e_{l(p-1)+m-1,l(q-1)+1}^{(r)}) = 0$.
Also, if $m\geq r+2$, then
$\pi(f_{l(p-1)+m,l(q-1)+1}^{(r+1)}) = \pi(f_{l(p-1)+m-1,l(q-1)+1}^{(r)}) = 0$.
Hence (\ref{equG3}) holds for $r+1$.
\end{proof}

\begin{lemma}\label{Lem2}
Let $1\leq m\leq l$, and $1\leq p, q, \leq {n\over l}$. Then
\begin{align}\label{equL1}
&\pi(e_{l(p-1) + m,l(q-1) + 1}^{(m)}) =
\sum_{k = 1}^{m}\pi(e_{l(p-1) + k, l(q-1)+ k}),\\
&\pi(f_{l(p-1)+m,l(q-1)+1}^{(m)}) =
\sum_{k = 1}^{m}(-1)^{k-1}\pi(f_{l(p-1)+k, l(q-1)+k}).
\nonumber
\end{align}
\end{lemma}

\begin{proof}
We proceed by induction on $m$. If $m = 1$, then (\ref{equL1})  obviously holds.
Assume that (\ref{equL1}) holds for $m$.
By (\ref{equG1}) and (\ref{equG2})
we have that
\begin{align}
& e_{l(p-1)+m+1,l(q-1)+1}^{(m+1)} =
\nonumber\\
&\sum_{s =1}^{n\over l}\Big[
\sum_{k = 1}^{l}e_{l(p-1)+m+1,l(s-1)+k}e_{l(s-1)+k,l(q-1)+1}^{(m)}
+ (-1)^m \sum_{k = 1}^{l}f_{l(p-1)+m+1,l(s-1)+k}f_{l(s-1)+k, l(q-1)+1}^{(m)}\Big] = \nonumber\\
&\sum_{s =1}^{n\over l}\Big[
\sum_{k = 1}^{m}
\Big(e_{l(s-1)+k,l(q-1) + 1}^{(m)}e_{l(p-1) + m+1,l(s-1)+k} +
e_{l(p-1)+m+1, l(q-1)+1}^{(m)}\Big) +\nonumber\\
&\sum_{k = m+1}^{l}e_{l(p-1)+m+1,l(s-1)+k}e_{l(s-1)+k, l(q-1)+1}^{(m)} +
(-1)^{m}\Big(\sum_{k = 1}^{m}(-f_{l(s-1)+k, l(q-1) +1}^{(m)}
f_{l(p-1) + m+1,l(s-1)+k} + \nonumber\\
&(-1)^{m+1}e_{l(p-1)+m+1, l(q-1)+1}^{(m)}) +
\sum_{k = m+1}^{l}f_{l(p-1)+m+1,l(s-1)+k}f_{l(s-1)+k, l(q-1)+1}^{(m)}\Big)\Big].
\nonumber
\end{align}
We also have
\begin{align}
&f_{l(p-1)+m+1, l(q-1)+1}^{(m+1)} =\nonumber\\
&\sum_{s =1}^{n\over l}\Big[
\sum_{k = 1}^{l}e_{l(p-1)+m+1,l(s-1)+k}f_{l(s-1)+k, l(q-1)+1}^{(m)}
+ (-1)^m \sum_{k = 1}^{l}f_{l(p-1)+m+1,l(s-1)+k}e_{l(s-1)+k, l(q-1)+1}^{(m)}\big] =  \nonumber\\
&\sum_{s =1}^{n\over l}\Big[
\sum_{k = 1}^{m}\Big(f_{l(s-1)+k,l(q-1)+1}^{(m)}e_{l(p-1)+m+1,l(s-1)+k} + f_{l(p-1)+m+1,l(q-1)+1}^{(m)}\Big) +
\nonumber\\
&\sum_{k = m+1}^{l}e_{l(p-1)+m+1,l(s-1)+k}f_{l(s-1)+k,l(q-1)+1}^{(m)} +
(-1)^{m}\Big(\sum_{k = 1}^{m}\Big(e_{l(s-1)+k,l(q-1)+1}^{(m)}f_{l(p-1)+m+1,l(s-1)+k} + \nonumber\\
&(-1)^{m+1}f_{l(p-1)+m+1, l(q-1)+1}^{(m)}\Big) +
\sum_{k = m+1}^{l}f_{l(p-1)+m+1,l(s-1)+k}e_{l(s-1)+k, l(q-1)+1}^{(m)}\Big)\Big].
\nonumber
\end{align}
Then using (\ref{equA1}) and (\ref{equG3})
we obtain that
\begin{align}
&\pi(e_{l(p-1)+ m+1,l(q-1)+1}^{(m+1)}) =
\pi(e_{l(p-1)+m,l(q-1)+1}^{(m)}) + \pi(e_{l(p-1)+m+1, l(q-1)+m+1}),
 \nonumber\\
&\pi(f_{l(p-1) + m+1,l(q-1) + 1}^{(m+1)}) =
 \pi(f_{l(p-1)+m,l(q-1)+1}^{(m)}) + (-1)^m\pi(f_{l(p-1)+m+1,l(q-1)+m+1}).
 \nonumber
\end{align}
By induction hypothesis we have
\begin{align}
&\pi(e_{l(p-1)+m+1, l(q-1)+1}^{(m+1)}) =
\sum_{k = 1}^{m}\pi(e_{l(p-1)+k, l(q-1)+k}) + \pi(e_{l(p-1)+m+1, l(q-1)+m+1}) = \nonumber\\
&\sum_{k = 1}^{m+1}\pi(e_{l(p-1) + k, l(q-1)+ k}),\nonumber\\
&\pi(f_{l(p-1)+m+1, l(q-1)+1}^{(m+1)}) =
\sum_{k = 1}^{m}(-1)^{k-1}\pi(f_{l(p-1)+k, l(q-1)+k})
 + (-1)^m\pi(f_{l(p-1)+m+1, l(q-1)+m+1}) = \nonumber\\
&\sum_{k = 1}^{m+1}(-1)^{k-1}\pi(f_{l(p-1)+k, l(q-1)+k}).
 \nonumber
\end{align}
\end{proof}

\noindent
Consider the Kazhdan filtration on $U(\p)$. By definition, the graded algebra
$Gr_K U(\p)$ is isomorphic to $S(\p)$. Moreover, $Gr_K U(\p)\simeq S(\p)$ is a commutative graded ring, where
the grading is induced from the Dynkin $\Z$-grading of $\g$.
For any $X\in U(\p)$ let $Gr_K(X)$ denote the corresponding element in $Gr_K U(\p)$ and
$P(X)$ denote the highest weight component of $Gr_K(X)$ in the Dynkin $\Z$-grading.
For $X\in U(\p)$, we denote by $\hbox{deg}P(X)$ the Kazhdan degree of $Gr_K(X)$ and by $\hbox{wt}P(X)$ the weight of the highest weight component
of  $Gr_K(X)$.

\begin{lemma}\label{Lem3}
\begin{align}\label{M1}
&P\Big(\pi(e_{lp, l(q-1) +1}^{(l+k)})\Big) =
\sum_{i=1}^{l-k}e_{l(p-1)+i,l(q-1)+i+k},
\end{align}
\begin{align}\label{M2}
&P\Big(\pi(f_{lp, l(q-1) +1}^{(l+k)})\Big) =
\sum_{i=1}^{l-k}(-1)^{i+k-1}f_{l(p-1)+i,l(q-1)+i+k},\\
&\hbox{where } k = 0, 1, \ldots, l-1.
 \nonumber
\end{align}
\end{lemma}

\begin{proof} We will prove a more general statement. We claim that for $0\leq k\leq l-1$ and $1\leq m\leq l$
\begin{align}\label{M3}
&P\Big(\pi(e_{l(p-1)+m, l(q-1) +1}^{(m+k)})\Big) =
\sum_{i=1}^{r}e_{l(p-1)+i,l(q-1)+i+k},
\end{align}
\begin{align}\label{M4}
&P\Big(\pi(f_{l(p-1)+m, l(q-1) +1}^{(m+k)})\Big) =
\sum_{i=1}^{r}(-1)^{i+k-1}f_{l(p-1)+i,l(q-1)+i+k},\\
&\hbox{where } r = \hbox{min}\lbrace m, l-k\rbrace.
 \nonumber
\end{align}
In particular,
\begin{align}
&\hbox{deg}P\Big(\pi(e_{l(p-1)+m, l(q-1) +1}^{(m+k)})\Big) =
\hbox{deg}P\Big(\pi(f_{l(p-1)+m, l(q-1) +1}^{(m+k)})\Big) = 2k+2,\nonumber\\
&\hbox{wt}P\Big(\pi(e_{l(p-1)+m, l(q-1) +1}^{(m+k)})\Big) =
\hbox{wt}P\Big(\pi(f_{l(p-1)+m, l(q-1) +1}^{(m+k)})\Big) = 2k.
 \nonumber
\end{align}

\noindent
We proceed to the proof of (\ref{M3}) by induction on $k$ and $m$.
Note that if $k = 0$, then (\ref{M3}) holds for any $1\leq m\leq l$ by (\ref{equL1}).
Assume that
if $k\leq b-1$, then (\ref{M3}) holds for any $1\leq m\leq l$.
Let $k = b$. Show that (\ref{M3}) holds for $m = 1$.
Note that by (\ref{equG1})
\begin{align}
&e_{l(p-1)+1, l(q-1) +1}^{(1+b)} =
\sum_{s =1}^{n\over l}
\Big[\Big(\sum_{a = 1}^l e_{l(p-1)+1,l(s-1)+a}e_{l(s-1)+a, l(q-1)+1}^{(b)}\Big)  +\nonumber\\
&(-1)^b\Big(\sum_{a = 1}^l f_{l(p-1)+1,l(s-1)+a}f_{l(s-1)+a, l(q-1)+1}^{(b)}\Big)\Big].
\nonumber
\end{align}

\noindent
Let $X = \pi(e_{l(p-1)+1,l(s-1)+a}e_{l(s-1)+a, l(q-1)+1}^{(b)})$,
$Y = \pi(f_{l(p-1)+1,l(s-1)+a}f_{l(s-1)+a, l(q-1)+1}^{(b)})$ where $a = 1, \ldots, b$.
Note that
\begin{align}
&\hbox{deg}P\Big(\pi(e_{l(p-1)+1,l(s-1)+a})\Big) =
\hbox{deg}P\Big(\pi(f_{l(p-1)+1,l(s-1)+a})\Big) = 2a,\nonumber\\
&\hbox{wt}P\Big(\pi(e_{l(p-1)+1,l(s-1)+a})\Big) = \hbox{wt}P\Big(\pi(f_{l(p-1)+1,l(s-1)+a})\Big) = 2a-2.
 \nonumber
\end{align}
By induction hypothesis,
\begin{align}
&\hbox{deg}P\Big(\pi(e_{l(s-1)+a, l(q-1)+1}^{(b)})\Big)
 = \hbox{deg}P\Big(\pi(f_{l(s-1)+a, l(q-1)+1}^{(b)})\Big)
  = 2(b-a)+2,\nonumber\\
&\hbox{wt}P\Big(\pi(e_{l(s-1)+a, l(q-1)+1}^{(b)})\Big) = \hbox{wt}P\Big(\pi(f_{l(s-1)+a, l(q-1)+1}^{(b)})\Big) = 2(b-a).
 \nonumber
\end{align}

\noindent
Then
\begin{align}
&\hbox{deg}P(X) = 2b+2, \quad \hbox{wt}P(X) = 2b-2,\nonumber\\
&\hbox{deg}P(Y) = 2b+2, \quad \hbox{wt}P(Y) = 2b-2.
\nonumber
\end{align}

\noindent
Let $X = \sum_{s =1}^{n\over l}\pi(e_{l(p-1)+1,l(s-1)+b+1}e_{l(s-1)+b+1, l(q-1)+1}^{(b)})$.
Then by (\ref{equG3}) $X  = \pi(e_{l(p-1)+1,l(q-1)+b+1}).$ Hence
$$\hbox{deg}P(X) = 2b+2, \quad \hbox{wt}P(X) = 2b.$$
Finally, by (\ref{equG3})
$$\pi(e_{l(p-1)+b+i, l(q-1)+1}^{(b)}) = 0\hbox{ for }i = 2, \ldots, l-b, \quad
\pi(f_{l(p-1)+b+i, l(q-1)+1}^{(b)}) = 0\hbox{ for }i = 1, \ldots, l-b.$$
Hence
$$P(\pi(e_{l(p-1)+1, l(q-1)+1}^{(1+b)})) = e_{l(p-1)+1, l(q-1)+b+1}.$$

\noindent
Let $k = b$ and assume that (\ref{M3}) holds for $m \leq c$.
Show that it holds for $m = c+1$. Note that (\ref{equG1})
\begin{align}
&e_{l(p-1)+c+1, l(q-1)+1}^{(c+1+b)} =
\sum_{s =1}^{n\over l}
\Big[\Big(
\sum_{a = 1}^{l}e_{l(p-1)+c+1,l(s-1)+a}e_{l(s-1)+a,l(q-1)+1}^{(c+b)}\Big) +\nonumber\\
&(-1)^{c+b}
\Big(\sum_{a = 1}^{l}f_{l(p-1)+c+1,l(s-1)+a}f_{l(s-1)+a,l(q-1)+1}^{(c+b)}\Big)
\Big].
\nonumber
\end{align}

\noindent
Thus
\begin{align}
&\pi(e_{l(p-1)+c+1, l(q-1)+1}^{(c+1+b)}) =
\sum_{s =1}^{n\over l}\Big[
\sum_{a = 1}^{c-1}\pi(e_{l(p-1)+c+1, l(s-1)+a}e_{l(s-1)+a,l(q-1)+1}^{(c+b)}) + \nonumber\\
&\pi(e_{l(p-1)+c+1,l(s-1)+c}e_{l(s-1)+c,l(q-1)+1}^{(c+b)})+
\sum_{i = 1}^{b}\pi(e_{l(p-1)+c+1,l(s-1)+c+i}e_{l(s-1)+c+i,l(q-1)+1}^{(c+b)}) +\nonumber\\
&\pi(e_{l(p-1)+c+1,l(s-1)+c+b+1}e_{l(s-1)+c+b+1,l(q-1)+1}^{(c+b)})+
\sum_{i = 2}^{l-c-b}\pi(e_{l(p-1)+c+1,l(s-1)+c+b+i}e_{l(s-1)+c+b+i,l(q-1)+1}^{(c+b)})+ \nonumber\\
&(-1)^{c+b}\Big(\sum_{a = 1}^{c}\pi(f_{l(p-1)+c+1,l(s-1)+a}f_{l(s-1)+a,l(q-1)+1}^{(c+b)}) +
\sum_{i = 1}^{b}\pi(f_{l(p-1)+c+1,l(s-1)+c+i}f_{l(s-1)+c+i,l(q-1)+1}^{(c+b)}) + \nonumber\\
&\sum_{i = 1}^{l-c-b}\pi(f_{l(p-1)+c+1,l(s-1)+c+b+i}f_{l(s-1)+c+b+i,l(q-1)+1}^{(c+b)})\Big)
\Big ].
\nonumber
\end{align}

\noindent
Let $X = \pi(e_{l(p-1)+c+1,l(s-1)+i}e_{l(s-1)+i,l(q-1)+1}^{(c+b)})$, where $i = 1, \ldots, c-1$, and

\noindent
$Y = \pi(f_{l(p-1)+c+1,l(s-1)+i}f_{l(s-1)+i,l(q-1)+1}^{(c+b)})$, where $i = 1, \ldots, c$.
 By (\ref{equG2}) and (\ref{equA1})
 \begin{align}
&X = \pi(e_{l(s-1)+i,l(q-1)+1}^{(c+b)}e_{l(p-1)+c+1,l(s-1)+i} +
e_{l(p-1)+c+1,l(q-1)+1}^{(c+b)}) =
\pi(e_{l(p-1)+c+1,l(q-1)+1}^{(c+b)}),\nonumber\\
&Y = \pi(-f_{l(s-1)+i,l(q-1)+1}^{(c+b)}f_{l(p-1)+c+1,l(s-1)+i} +
(-1)^{c+b+1}e_{l(p-1)+c+1,l(q-1)+1}^{(c+b)}) =\nonumber\\
&\pi((-1)^{c+b+1}e_{l(p-1)+c+1,l(q-1)+1}^{(c+b)}).
\nonumber
\end{align}

\noindent
By induction hypothesis
\begin{align}\label{E1}
&\hbox{deg}P(X) = \hbox{deg}P(Y) = 2b, \\
&\hbox{wt}P(X) = \hbox{wt}P(Y) = 2b-2.
\nonumber
\end{align}

\noindent
Let $X = \pi(e_{l(p-1)+c+1,l(s-1)+c}e_{l(s-1)+c,l(q-1)+1}^{(c+b)})$.
Then by (\ref{equG2}) and  (\ref{equA1})
\begin{align}
&X = \pi(e_{l(s-1)+c,l(q-1)+1}^{(c+b)}e_{l(p-1)+c+1,l(s-1)+c} + e_{l(p-1)+c+1,l(q-1)+1}^{(c+b)}) =\nonumber\\
&\pi(e_{l(p-1)+c,l(q-1)+1}^{(c+b)} + e_{l(p-1)+c+1,l(q-1)+1}^{(c+b)}).
\nonumber
\end{align}

\noindent
By induction hypothesis
\begin{align}\label{E2}
&\hbox{deg}P(\pi(e_{l(p-1)+c,l(q-1)+1}^{(c+b)})) = 2b+2, \\
&\hbox{wt}P(\pi(e_{l(p-1)+c,l(q-1)+1}^{(c+b)})) = 2b,
\nonumber
\end{align}

\begin{align}\label{E3}
&\hbox{deg}P(\pi(e_{l(p-1)+c+1,l(q-1)+1}^{(c+b)})) = 2b,\\
&\hbox{wt}P(\pi(e_{l(p-1)+c+1,l(q-1)+1}^{(c+b)})) = 2b-2.
\nonumber
\end{align}

\noindent
Let $X = \pi(e_{l(p-1)+c+1,l(s-1)+c+i}e_{l(s-1)+c+i,l(q-1)+1}^{(c+b)})$,
$Y = \pi(f_{l(p-1)+c+1,l(s-1)+c+i}f_{l(s-1)+c+i,l(q-1)+1}^{(c+b)})$
for $i = 1, \ldots, b$.
Then by induction hypothesis

\begin{align}\label{E4}
&\hbox{deg}P(X) = \hbox{deg}P(Y) =2b+2, \\
&\hbox{wt}P(X) = \hbox{wt}P(Y) = 2b-2.
\nonumber
\end{align}

\noindent
Let $X = \pi(e_{l(p-1)+c+1,l(s-1)+c+b+1}e_{l(s-1)+c+b+1,l(q-1)+1}^{(c+b)})$.
Hence by (\ref{equG3}) $X = \pi(e_{l(p-1)+c+1,l(q-1)+c+b+1})$. Then
\begin{align}\label{E5}
&\hbox{deg}P(X) = 2b+2, \\
\nonumber
&\hbox{wt}P(X) = 2b.
\end{align}
Finally, by (\ref{equG3})  $\pi(e_{l(p-1)+c+1, l(s-1)+c+b+i}e_{l(s-1)+c+b+i,l(q-1)+1}^{(c+b)}) = 0$ for $i = 2,\ldots, l-c-b$ and $\pi(f_{l(p-1)+c+1, l(s-1)+c+b+i}f_{l(s-1)+c+b+i,l(q-1)+1}^{(c+b)}) = 0$ for $i = 1,\ldots,l-c-b$.
From (\ref{E1})-(\ref{E5}) one can see that the highest degree component in
$\pi(e_{l(p-1)+c+1,l(q-1)+1}^{(c+1+b)})$ has degree $2b+2$,
and its highest weight component has weight $2b$. In fact, if $c\geq l-b$, then by (\ref{E2})
this component is  $P(\pi(e_{l(p-1)+c,l(q-1)+1}^{(c+b)}))$. By induction hypothesis
$P(\pi(e_{l(p-1)+c,l(q-1)+1}^{(c+b)})) = \sum_{i= 1}^{l-b}e_{l(p-1)+ i,l(q-1)+i+b}$.
If  $c < l-b$, then $P(\pi(e_{l(p-1)+c,l(q-1)+1}^{(c+b)})) = \sum_{i= 1}^{c}e_{l(p-1)+i,l(q-1)+i+b}$. Note that in this case
$\pi(e_{l(p-1)+c+1,l(q-1)+1}^{(c+1+b)})$ has
an additional element $\pi(e_{l(p-1)+c+1,l(q-1)+ c+b+1})$
of degree $2b + 2$ and weight $2b$ according to (\ref{E5}).
Clearly,  $P(\pi(e_{l(p-1)+c+1,l(q-1)+ c+b+1})) = e_{l(p-1)+c+1,l(q-1)+ c+b+1}$ and
$P(\pi(e_{l(p-1)+c,l(q-1)+1}^{(c+b)})) +  P(\pi(e_{l(p-1)+c+1,l(q-1)+ c+b+1}))  \not= 0$.
Hence
\begin{align}
&P(\pi(e_{l(p-1)+c+1,l(q-1)+1}^{(c+1+b)}))= P(\pi(e_{l(p-1)+c,l(q-1)+1}^{(c+b)})) +
P(\pi(e_{l(p-1)+c+1,l(q-1)+ c+b+1})) = \nonumber\\
&\sum_{i= 1}^{c+1}e_{l(p-1)+i,l(q-1)+i+b}.
\nonumber
\end{align}
Then in either case,
$$P(\pi(e_{l(p-1)+c+1,l(q-1)+1}^{(c+1+b)}))  = \sum_{i= 1}^{r}e_{l(p-1)+i,l(q-1)+i+b},\hbox{ where }
r = \hbox{min}\lbrace c+1, l-b\rbrace.$$
Thus if $0\leq k\leq l-1$ and $1\leq m\leq l$, then
(\ref{M3}) holds.
Similarly, one can prove that (\ref{M4}) holds.
In particular, if $m = l$ and $k = 0, \ldots, l-1$, we obtain
(\ref{M1}) and (\ref{M2}).

\end{proof}

\vskip 0.1in
The statement of Theorem \ref{T2} follows from Lemma \ref{Lem3} and Theorem \ref{maint} (a).
Lemma \ref{Lem3} and Theorem \ref{maint} (b) imply the following

\vskip 0.1in
\begin{corollary}\label{cor}

$${Gr}_KW_\chi\simeq S(\g^{\chi}).$$

\end{corollary}

\noindent
This implies that Conjecture 2.8 in \cite{PS2} is true in this case.

\noindent
Let $$x^i_{p,q} = e_{l(p-1)+i, l(q-1) +i}, \quad
\xi^i_{p,q} = (-1)^{i+1}f_{l(p-1)+i, l(q-1) +i}$$
for $p, q = 1, \ldots, {n\over l}$ and $i = 1, \ldots, l$.
Then

\begin{align}
&\g_0 = <x^i_{p,q} \hbox{ }|\hbox{ }\xi^i_{p,q}>,
\nonumber\\
&\hbox{where } i = 1, \ldots, l, \hbox{ and } 1\leq p,q \leq  {n\over l}.
\nonumber
\end{align}

\begin{theorem}\label{T3}
\begin{align}
&\vartheta(\pi(e_{lp,l(q-1)+1}^{(l+k-1)})) = [\sum_{1\leq p_1, p_2, \ldots, p_{k-1}\leq{n\over l}} \nonumber\\
&\sum_{l\geq i_1\geq i_2\geq\ldots \geq i_k\geq 1}
(x^{i_1}_{p, p_1}  + (-1)^{k+1}\xi^{i_1}_{p, p_1})
(x^{i_2}_{p_1, p_2} + (-1)^{k}\xi^{i_2}_{p_1, p_2})
\ldots
(x^{i_{k}}_{p_{k-1}, q} + \xi^{i_{k}}_{p_{k-1}, q})]_{even},\nonumber\\
&\vartheta(\pi(f_{lp,l(q-1)+1}^{(l+k-1)})) =
[\sum_{1\leq p_1, p_2, \ldots, p_{k-1}\leq{n\over l}} \nonumber\\
&\sum_{l\geq i_1\geq i_2\geq\ldots \geq i_k\geq 1}
(x^{i_1}_{p, p_1}  + (-1)^{k+1}\xi^{i_1}_{p, p_1})
(x^{i_2}_{p_1, p_2} + (-1)^{k}\xi^{i_2}_{p_1, p_2})
\ldots
(x^{i_{k}}_{p_{k-1}, q} + \xi^{i_{k}}_{p_{k-1}, q})]_{odd}
\nonumber
\end{align}
\end{theorem}

\begin{proof}
\noindent
We will prove by induction on $s$ and $r$ that if $s\geq 0$ and $1\leq r \leq l$, then
\begin{align}\label{P1}
&\vartheta(\pi(e_{l(p-1) + r,l(q-1) + 1}^{(r+s)})) + \vartheta(\pi(f_{l(p-1) + r,l(q-1)+1}^{(r+s)})) = \\
&\sum_{1\leq p_1, p_2, \ldots, p_{s}\leq{n\over l}}
[\sum_{r\geq i_1\geq i_2\geq\ldots \geq i_{s+1}\geq 1}
(x^{i_1}_{p, p_1} + (-1)^{s}\xi^{i_1}_{p, p_1}) \ldots
(x^{i_{s+1}}_{p_s, q}  + \xi^{i_{s+1}}_{p_s, q})].
\nonumber
\end{align}
Note that if $s = 0$, then (\ref{P1}) holds for any $1\leq r \leq l$, since
\begin{align}
&\vartheta(\pi(e_{l(p-1)+r,l(q-1)+1}^{(r)})) +
\vartheta(\pi(f_{l(p-1)+r,l(q-1)+1}^{(r)})) = \nonumber\\
&\vartheta(\sum_{i= 1}^r\pi(e_{l(p-1)+i,l(q-1)+i})) +
\vartheta(\sum_{i= 1}^r(-1)^{i-1}\pi(f_{l(p-1)+i,l(q-1)+i})) =
\sum_{i= 1}^r(x^i_{p,q} + \xi^i_{p,q}).
\nonumber
\end{align}
Assume that if $s\leq k-1$, then (\ref{P1}) holds for any $1\leq r \leq l$.
Let $s = k$, show that (\ref{P1}) holds for $r = 1$.
We  have
\begin{align}
&\vartheta(\pi(e_{l(p-1)+1,l(q-1)+1}^{(1+k)})) + \vartheta(\pi(f_{l(p-1)+1,l(q-1)+1}^{(1+k)})) = \nonumber\\
&\sum_{p_1=1}^{n\over l}
\vartheta(\pi(e_{l(p-1)+1,l(p_1-1)+1)})\pi(e_{l(p_1 - 1)+1,l(q-1)+1}^{(k)}) + (-1)^k\pi(f_{l(p-1)+1,l(p_1 -1)+1})\pi(f_{l(p_1 - 1)+1,l(q-1)+1}^{(k)})) + \nonumber\\
&\sum_{p_1=1}^{n\over l}
\vartheta(\pi(e_{l(p-1)+1,l(p_1 - 1)+1})\pi(f_{l(p_1 - 1)+1,l(q-1) + 1}^{(k)}) + (-1)^k\pi(f_{l(p-1)+1,l(p_1-1)+1})\pi(e_{l(p_1-1)+1,l(q-1)+1}^{(k)}))=
\nonumber
\end{align}
\begin{align}
&\sum_{p_1=1}^{n\over l}
(e_{l(p-1)+1,l(p_1-1)+1} + (-1)^kf_{l(p-1)+1,l(p_1-1)+1})
(\vartheta(\pi(e_{l(p_1 - 1)+1,l(q-1) + 1}^{(k)}) +
\vartheta(\pi(f_{l(p_1 - 1)+1,l(q-1)+1}^{(k)})) = \nonumber\\
&\sum_{p_1=1}^{n\over l}
(x^1_{p, p_1} +  (-1)^k\xi^1_{p, p_1})
\sum_{1\leq p_2, \ldots, p_{k}\leq {n\over l}}
(x^1_{p_1, p_2} +  (-1)^{k-1}\xi^1_{p_1, p_2})\ldots
(x^1_{p_k, q} + \xi^1_{p_k, q}) = \nonumber\\
&\sum_{l\leq p_1, \ldots, p_k\leq {n\over l}}
(x^1_{p, p_1} +  (-1)^{k}\xi^1_{p, p_1})\ldots
(x^1_{p_k, q} + \xi^1_{p_k, q}).
\nonumber
\end{align}
Let $s = k$ and assume that (\ref{P1}) holds for $r\leq m$.
Show that it holds for $r = m+1$.
By induction hypothesis we have

\begin{align}
&\vartheta(\pi(e_{l(p-1)+m+1,l(q-1)+1}^{(m+1+k)})) + \vartheta(\pi(f_{l(p-1)+m+1,l(q-1)+1}^{(m+1+k)})) = \nonumber\\
&\vartheta(\pi(e_{l(p-1)+m+1,l(p-1)+m}))
\vartheta(\pi(e_{l(p-1)+m,l(q-1)+1}^{(m+k)})) + \nonumber\\
&\sum_{p_1=1}^{n\over l}
\Big(\vartheta(\pi(e_{l(p-1)+m+1,l(p_1-1)+m+1}))
\vartheta(\pi(e_{l(p_1-1)+m+1,l(q-1)+1}^{(m+k)}))\Big) +
\nonumber
\end{align}
\begin{align}
&(-1)^{m+k}
\sum_{p_1=1}^{n\over l}
\Big(\vartheta(\pi(f_{l(p-1)+m+1,l(p_1-1)+m+1}))
\vartheta(\pi(f_{l(p_1-1)+m+1,l(q-1)+1}^{(m+k)}))\Big) +\nonumber\\
&\vartheta(\pi(e_{l(p-1)+m+1,l(p-1)+m}))
\vartheta(\pi(f_{l(p-1)+m,l(q-1)+1}^{(m+k)})) + \nonumber\\
&\sum_{p_1=1}^{n\over l}
\Big(\vartheta(\pi(e_{l(p-1)+m+1,l(p_1-1)+m+1}))
\vartheta(\pi(f_{l(p_1-1)+m+1,l(q-1)+1}^{(m+k)}))\Big) + \nonumber\\
&(-1)^{m+k}
\sum_{p_1=1}^{n\over l}
\Big(\vartheta(\pi(f_{l(p-1)+m+1,l(p_1-1)+m+1}))
\vartheta(\pi(e_{l(p_1-1)+m+1,l(q-1)+1}^{(m+k)}))\Big) =
\nonumber
\end{align}
\begin{align}
&\vartheta(\pi(e_{l(p-1)+m,l(q-1)+1}^{(m+k)})) +
\vartheta(\pi(f_{l(p-1)+m,l(q-1)+1}^{(m+k)})) + \nonumber\\
&\sum_{p_1=1}^{n\over l}
\Big([\vartheta(\pi(e_{l(p-1)+m+1,l(p_1-1)+m+1})) +
(-1)^{m+k}\vartheta(\pi(f_{l(p-1)+m+1,l(p_1-1)+m+1}))] \times\nonumber\\
&[\vartheta(\pi(e_{l(p_1-1)+m+1,l(q-1)+1}^{(m+k)})) +
\vartheta(\pi(f_{l(p_1-1)+m+1,l(q-1)+1}^{(m+k)}))]\Big) = \nonumber\\
&\sum_{l\leq p_1, \ldots, p_k\leq {n\over l}}
\Big(\sum_{m\geq i_1\geq i_2\geq\ldots\geq i_{k+1}\geq 1}
(x^{i_1}_{p, p_1} + (-1)^{k}\xi^{i_1}_{p, p_1})\ldots
(x^{i_{k+1}}_{p_k, q} + \xi^{i_{k+1}}_{p_k, q})\Big) + \nonumber\\
&\sum_{p_1=1}^{n\over l}(x^{m+1}_{p, p_1} + (-1)^{k}\xi^{m+1}_{p, p_1})
\sum_{l\leq p_2, \ldots, p_k\leq {n\over l}}
\Big(\sum_{m+1\geq i_1\geq i_2\geq\ldots\geq i_{k}\geq 1}
(x^{i_1}_{p_1, p_2} + (-1)^{k-1}\xi^{i_1}_{p_1, p_2})\ldots
(x^{i_k}_{p_k, q} + \xi^{i-k}_{p_k, q})\Big) =
\nonumber\\
&\sum_{l\leq p_1, \ldots, p_k\leq {n\over l}}
\Big(\sum_{m+1\geq i_1\geq i_2\geq\ldots\geq i_{k+1}\geq 1}
(x^{i_1}_{p, p_1} + (-1)^{k}\xi^{i_1}_{p, p_1})\ldots
(x^{i_{k+1}}_{p_k, q} + \xi^{i_{k+1}}_{p_k,q})\Big).
\nonumber
\end{align}
Thus (\ref{P1}) holds. In particular, if $r = l$ and $s = k-1$,
we obtain the statement of Theorem \ref{T3}.
\end{proof}

\noindent
At the moment we are interested in $Y(Q({n\over l}))$.
\begin{theorem}\label{T4}
 Let $ 1\leq p, q\leq {n\over l}$.
\begin{align}
&U^{\otimes l}\circ \Delta_l^{op}(T_{q,p}^{(r)})
= (-1)^r[\sum_{1\leq p_1, p_2, \ldots, p_{r-1}\leq{n\over l}} \nonumber\\
&\sum_{l\geq i_1\geq i_2\geq\ldots \geq i_r\geq 1}
(x^{i_1}_{p, p_1}  + (-1)^{r+1}\xi^{i_1}_{p, p_1})
(x^{i_2}_{p_1, p_2} + (-1)^{r}\xi^{i_2}_{p_1, p_2})
\ldots
(x^{i_{r}}_{p_{r-1}, q} + \xi^{i_{r}}_{p_{r-1}, q})]_{even},\nonumber\\
&U^{\otimes l}\circ \Delta_l^{op}(T_{-q,p}^{(r)}) =
(-1)^r [\sum_{1\leq p_1, p_2, \ldots, p_{r-1}\leq{n\over l}} \nonumber\\
&\sum_{l\geq i_1\geq i_2\geq\ldots \geq i_r\geq 1}
(x^{i_1}_{p, p_1}  + (-1)^{r+1}\xi^{i_1}_{p, p_1})
(x^{i_2}_{p_1, p_2} + (-1)^{r}\xi^{i_2}_{p_1, p_2})
\ldots
(x^{i_{r}}_{p_{r-1}, q} + \xi^{i_{r}}_{p_{r-1}, q})]_{odd}
\nonumber
\end{align}
\end{theorem}

\begin{proof}

\noindent
Let $1\leq p, q\leq {n\over l}$.
Set
\begin{align}\label{X}
& {T}_{q,p}^{(r)+}=T_{q,p}^{(r)}+T_{-q,p}^{(r)},\nonumber\\
& {T}_{q,p}^{(r)-}=T_{q,p}^{(r)}-T_{-q,p}^{(r)}.
\nonumber
\end{align}
One can easily verify the following recursive relations
\begin{equation}\label{X1}
\Delta^{op}_l({T}_{q,p}^{(r)+})=\sum_{s=0}^r
\Big( \sum_{k = 1}^{n\over l}
\left(T_{k,p}^{(r-s)}+
(-1)^sT_{-k,p}^{(r-s)}\right)\otimes \Delta^{op}_{l-1}\left({T}_{q,k}^{(s)+}\right)\Big).
\end{equation}
As a direct consequence of (\ref{equG1}) we obtain
\begin{align}\label{X2}
&U({T}_{q,p}^{(r)+})=(-1)^r\sum_{1\leq p_1, \ldots, p_{r-1}\leq {n\over l}}\\
&\left(e_{p,p_1}+(-1)^{r-1}f_{p,p_1}\right)
\left(e_{p_1,p_2}+(-1)^{r-2}f_{p_1,p_2}\right)\ldots
\left(e_{p_{r-1},q}+ f_{p_{r-1},q}\right),\nonumber
\end{align}
\begin{align}\label{X3}
&U({T}_{q,p}^{(r)-})=
(-1)^r\sum_{1\leq p_1, \ldots, p_{r-1}\leq {n\over l}}\\
&\left(e_{p,p_1}+(-1)^{r}f_{p,p_1}\right)
\left(e_{p_1,p_2}+(-1)^{r-1}f_{p_1,p_2}\right)\ldots
\left(e_{p_{r-1},q} - f_{p_{r-1},q}\right).\nonumber
\end{align}

\begin{lemma}\label{L4}
Identify  $U(\g_0)\subset U(Q(n))$ with $U(Q({n\over l}))^{\otimes l}$ by setting
\begin{align}
&x_{p,q}^i\mapsto 1^{\otimes l-i}\otimes e_{p,q}\otimes 1^{\otimes i-1},\nonumber\\
&\xi_{p,q}^i\mapsto 1^{\otimes l-i}\otimes f_{p,q}\otimes 1^{\otimes i-1}.
\nonumber
\end{align}
Then
\begin{align}
&U^{\otimes l}\circ\Delta_l^{op}({T}_{q,p}^{(r)+})=
(-1)^r[\sum_{1\leq p_1, p_2, \ldots, p_{r-1}\leq{n\over l}} \nonumber\\
&\sum_{l\geq i_1\geq i_2\geq\ldots \geq i_r\geq 1}
(x^{i_1}_{p, p_1}  + (-1)^{r+1}\xi^{i_1}_{p, p_1})
(x^{i_2}_{p_1, p_2} + (-1)^{r}\xi^{i_2}_{p_1, p_2})
\ldots
(x^{i_{r}}_{p_{r-1}, q} + \xi^{i_{r}}_{p_{r-1}, q})].
\nonumber
\end{align}
\end{lemma}

\begin{proof} Follows from (\ref{X1}), (\ref{X2}) and (\ref{X3}).
\end{proof}

\noindent
This completes the proof of Theorem \ref{T4}.
\end{proof}

\begin{corollary}\label{corol}
There exists a surjective homomorphism:
$$\varphi: Y(Q({n\over l}))\longrightarrow W_{\chi}$$
defined as follows:
\begin{equation}\label{equHC}
\varphi(T_{q,p}^{(r)}) = (-1)^r\pi(e_{lp,l(q-1)+1}^{(l+r-1)}),\quad
\varphi(T_{-q,p}^{(r)}) = (-1)^r\pi(f_{lp,l(q-1)+1}^{(l+r-1)}), \hbox{ for } r = 1, 2, \ldots.
\nonumber
\end{equation}
\end{corollary}

\begin{proof} Recall that the Harish-Chandra homomorphism $\vartheta: W_{\chi}\to U(\g_0)$ is injective (\cite{PS2}).
We have
$$(-1)^r\vartheta(\pi(e_{lp,l(q-1)+1}^{(l+r-1)}))=U^{\otimes l}\circ\Delta^{op}_l(T_{q,p}^{(r)}), \quad (-1)^r\vartheta(\pi(f_{lp,l(q-1)+1}^{(l+r-1)}))=U^{\otimes l}\circ\Delta^{op}_l(T_{-q,p}^{(r)}).$$
Hence $\varphi=\vartheta^{-1}\circ U^{\otimes l}\circ \Delta^{op}_l$ is a surjective homomorphism $\varphi: Y(Q({n\over l}))\longrightarrow W_{\chi}$.
\end{proof}

\noindent
This proves Theorem  \ref{T1}.

\begin{theorem}\label{ton}
\begin{equation}\label{kl}
U^{\otimes k}\circ \Delta^{op}_k (Y(Q({n\over l}))) =
\hbox{ev}^{\otimes k}\circ \Delta_k (Y(Q({n\over l}))).
\end{equation}
\end{theorem}
\begin{proof}
First, we will prove the following
\begin{lemma}\label{Lem7}
\begin{equation}\label{bone}
{\bar{ev}}\circ S({T}^{(r)+}_{q, p}) = U({T}^{(r)+}_{q, p}).
\end{equation}
\end{lemma}

\begin{proof}
According to (\ref{antiS}) and (\ref{antieval}) we have that
\begin{equation}\label{anti}
\bar{ev}\circ S (T_{p,q}^{(r)}) = (-1)^r
\bar{ev}((\sum_{i_1, \ldots i_{r-1}\in {\lbrace \pm 1, \ldots \pm n\rbrace}} (-1)^{\nu (p,i_1,\ldots, i_{r-1},q)}T_{p,i_1}^{(1)}T_{i_1,i_2}^{(1)}\ldots T_{i_{r-1},q}^{(1)})).
\end{equation}

\noindent
We proceed by induction on $r$. The statement is obviously true if $r = 1$.
Assume that (\ref{bone}) holds for $r$. Then according to (\ref{anti})
\begin{align}
&{\bar{ev}}\circ S({T}^{(r+1)}_{q, p}) =
(-1)^{r+1}\bar{ev}(\sum_{i_1,\ldots, i_{r}\in \lbrace \pm 1, \ldots, \pm
{n\over l}\rbrace}
(-1)^{\nu(q, i_1, \ldots, i_r, p)}
{T}^{(1)}_{q, i_1}{T}^{(1)}_{i_1, i_2}\ldots {T}^{(1)}_{i_{r-1}, i_r}{T}^{(1)}_{i_{r}, p}) =\nonumber\\
&(-1)^{r+1}\sum_{i_r = 1}^{n\over l}\bar{ev}
(\sum_{i_1,\ldots, i_{r-1}\in \lbrace \pm 1, \ldots,
\pm {n\over l}\rbrace}(-1)^{\nu(q, i_1, \ldots, i_r)}
{T}^{(1)}_{q, i_1}{T}^{(1)}_{i_1, i_2}\ldots {T}^{(1)}_{i_{r-1}, i_r}
{T}^{(1)}_{i_{r}, p}) +\nonumber\\
&(-1)^{r+1}\sum_{i_r = 1}^{n\over l}\bar{ev}
(\sum_{i_1,\ldots, i_{r-1}\in \lbrace \pm 1, \ldots,
\pm {n\over l}\rbrace}(-1)^{\nu(q, i_1, \ldots, -i_r)+1}
{T}^{(1)}_{q, i_1}{T}^{(1)}_{i_1, i_2}\ldots {T}^{(1)}_{i_{r-1}, -i_r}
{T}^{(1)}_{-i_{r}, p}) =\nonumber
\end{align}
\begin{align}
&(-1)^{r+1}\sum_{i_r = 1}^{n\over l}\bar{ev}({T}^{(1)}_{i_{r}, p})
\bar{ev}(\sum_{i_1,\ldots, i_{r-1}\in \lbrace \pm 1, \ldots, \pm {n\over l}\rbrace}(-1)^{\nu(q, i_1, \ldots, i_r)}
{T}^{(1)}_{q, i_1}{T}^{(1)}_{i_1, i_2}\ldots {T}^{(1)}_{i_{r-1}, i_r})
 +\nonumber\\
 &(-1)^{r+1}\sum_{i_r = 1}^{n\over l}\bar{ev}({T}^{(1)}_{-i_{r}, p})
\bar{ev}(\sum_{i_1,\ldots, i_{r-1}\in \lbrace \pm 1, \ldots, \pm {n\over l}\rbrace}(-1)^{\nu(-q, -i_1, \ldots, i_r)}
{T}^{(1)}_{-q, -i_1}{T}^{(1)}_{-i_1, -i_2}\ldots {T}^{(1)}_{-i_{r-1}, i_r})
(-1)^r = \nonumber\\
 &(-1)^{r+1}\sum_{i_r = 1}^{n\over l}
e_{p,i_r}({\bar{ev}}\circ S)({T}^{(r)}_{q, i_r})(-1)^r +
f_{p,i_r}({\bar{ev}}\circ S)({T}^{(r)}_{-q, i_r}).\nonumber
\end{align}
Similarly,
\begin{align}
&{\bar{ev}}\circ S({T}^{(r+1)}_{-q, p}) =
(-1)^{r+1}\bar{ev}(\sum_{i_1,\ldots, i_{r}\in \lbrace \pm 1, \ldots, \pm
{n \over l}\rbrace}
(-1)^{\nu(-q, i_1, \ldots, i_r, p)}
{T}^{(1)}_{-q, i_1}{T}^{(1)}_{i_1, i_2}\ldots {T}^{(1)}_{i_{r-1}, i_r}{T}^{(1)}_{i_{r}, p}) =\nonumber\\
&(-1)^{r+1}\sum_{i_r = 1}^{n\over l}\bar{ev}
(\sum_{i_1,\ldots, i_{r-1}\in \lbrace \pm 1, \ldots, \pm {n \over l}\rbrace}(-1)^{\nu(-q, i_1, \ldots, i_r)}
{T}^{(1)}_{-q, i_1}{T}^{(1)}_{i_1, i_2}\ldots {T}^{(1)}_{i_{r-1}, i_r}
{T}^{(1)}_{i_{r}, p}) +\nonumber\\
&(-1)^{r+1}\sum_{i_r = 1}^{n\over l}\bar{ev}
(\sum_{i_1,\ldots, i_{r-1}\in \lbrace \pm 1, \ldots, \pm {n \over l}\rbrace}(-1)^{\nu(-q, i_1, \ldots, -i_r)}
{T}^{(1)}_{-q, i_1}{T}^{(1)}_{i_1, i_2}\ldots {T}^{(1)}_{i_{r-1}, -i_r}
{T}^{(1)}_{-i_{r}, p}) =\nonumber
\end{align}
\begin{align}
&(-1)^{r+1}\sum_{i_r = 1}^{n\over l}\bar{ev}({T}^{(1)}_{i_{r}, p})
\bar{ev}(\sum_{i_1,\ldots, i_{r-1}\in \lbrace \pm 1, \ldots, \pm {n \over l}\rbrace}(-1)^{\nu(-q, i_1, \ldots, i_r)}
{T}^{(1)}_{-q, i_1}{T}^{(1)}_{i_1, i_2}\ldots {T}^{(1)}_{i_{r-1}, i_r})
 +\nonumber\\
 &(-1)^{r+1}\sum_{i_r = 1}^{n\over l}\bar{ev}({T}^{(1)}_{-i_{r}, p})
\bar{ev}(\sum_{i_1,\ldots, i_{r-1}\in \lbrace \pm 1, \ldots, \pm {n \over l}\rbrace}(-1)^{\nu(q, -i_1, \ldots, i_r)}
{T}^{(1)}_{q, -i_1}{T}^{(1)}_{-i_1, -i_2}\ldots {T}^{(1)}_{-i_{r-1}, i_r})
(-1)^r = \nonumber\\
 &(-1)^{r+1}\sum_{i_r = 1}^{n\over l}
e_{p,i_r}({\bar{ev}}\circ S)({T}^{(r)}_{-q, i_r})(-1)^r +
f_{p,i_r}({\bar{ev}}\circ S)({T}^{(r)}_{q, i_r}).\nonumber
\end{align}
Hence
\begin{align}
&{\bar{ev}}\circ S({T}^{(r+1)+}_{q, p}) =
(-1)^{r+1}\sum_{i_r = 1}^{n\over l}
e_{p,i_r}({\bar{ev}}\circ S)({T}^{(r)+}_{q, i_r})(-1)^r +
f_{p,i_r}({\bar{ev}}\circ S)({T}^{(r)+}_{q, i_r}) = \nonumber\\
&(-1)^{r+1}\sum_{i_r = 1}^{n\over l}
(e_{p,i_r} + (-1)^r f_{p,i_r})((-1)^rU({T}^{(r)+}_{q, i_r})
 = \nonumber\\
 &(-1)^{r+1}\sum_{i_r = 1}^{n\over l}
(e_{p,i_r} + (-1)^r f_{p,i_r})
 \sum_{1\leq p_1, \ldots, p_{r-1}\leq {n\over l}}
 (e_{i_r, p_1} + (-1)^{r-1} f_{i_r,p_1})\ldots
 (e_{p_{r-1},q} + f_{p_{r-1},q}) = U({T}^{(r+1)+}_{q, p}).
 \nonumber
\end{align}
\end{proof}

\noindent
It follows from Lemma \ref{Lem7} that
$$\bar{ev}^{\otimes k}\circ S^{\otimes k}\otimes \Delta^{op}_k ({T}^{(r)+}_{q, p}) =
U^{\otimes k}\circ \Delta^{op}_k ({T}^{(r)+}_{q, p}).$$
Finally, observe that the following diagram, where $Y: = Y(Q({n\over l}))$  is commutative:
$$\begin{CD}
Y @>\Delta>> Y\otimes Y @>id\circ\Delta>>  Y\otimes Y\otimes Y @>id\circ id\circ\Delta>>\ldots\\
@A{S}AA @A{S\otimes S}AA @A{S\otimes S\otimes S}AA@A{S^{\otimes 4}}AA  \\
Y @>\Delta^{op}>>Y\otimes Y @>\Delta^{op}\circ id>>Y\otimes Y\otimes Y
@>\Delta^{op}\circ id\circ id>>\ldots
\end{CD}
$$
Hence
\begin{equation}\label{last}
\bar{ev}^{\otimes k}\circ  \Delta_k\circ S
({T}^{(r)+}_{q, p}) =
U^{\otimes k}\circ \Delta^{op}_k ({T}^{(r)+}_{q, p}).
\end{equation}

\noindent
This completes the proof of Theorem \ref{ton}.
\end{proof}
\begin{corollary}\label{COR}
$$W_{\chi}\cong \bar{ev}^{\otimes l}\circ \Delta_l (Y(Q({n\over l})))$$
\end{corollary}
\begin{proof}
Follows from Theorem \ref{T1} and (\ref{kl}) where $k = l$.
\end{proof}
\begin{theorem}\label{TTTT}
 Let $ 1\leq p, q\leq {n\over l}$.
\begin{align}
&\hbox{ev}^{\otimes l}\circ \Delta_l(T_{q,p}^{(r)})
= (-1)^r[\sum_{1\leq p_1, p_2, \ldots, p_{r-1}\leq{n\over l}} \nonumber\\
&\sum_{1\leq i_1 < i_2 <\ldots < i_r\leq l}
(x^{i_1}_{p, p_1}  + (-1)^{r+1}\xi^{i_1}_{p, p_1})
(x^{i_2}_{p_1, p_2} + (-1)^{r}\xi^{i_2}_{p_1, p_2})
\ldots
(x^{i_{r}}_{p_{r-1}, q} + \xi^{i_{r}}_{p_{r-1}, q})]_{even},\nonumber\\
&\hbox{ev}^{\otimes l}\circ \Delta_l(T_{-q,p}^{(r)})
= (-1)^r[\sum_{1\leq p_1, p_2, \ldots, p_{r-1}\leq{n\over l}} \nonumber\\
&\sum_{1\leq i_1 < i_2 <\ldots < i_r\leq l}
(x^{i_1}_{p, p_1}  + (-1)^{r+1}\xi^{i_1}_{p, p_1})
(x^{i_2}_{p_1, p_2} + (-1)^{r}\xi^{i_2}_{p_1, p_2})
\ldots
(x^{i_{r}}_{p_{r-1}, q} + \xi^{i_{r}}_{p_{r-1}, q})]_{odd}
\nonumber
\end{align}
\end{theorem}
\begin{proof}
According to (\ref{delta}), if $1\leq m\leq l$, then
\begin{align}
&\Delta_m(T_{q,p}^{(r)})=\sum_{s=0}^r\sum_{k=1}^{n\over l} (-1)^{(p(q)+p(k))(p(p)+p(k))}
\Delta_{m-1}(T_{q,k}^{(s)})\otimes T_{k,p}^{(r-s)} = \nonumber\\
&\sum_{s=0}^r\sum_{k=1}^{n\over l}\Big(1^{\otimes m-1}\otimes T_{k,p}^{(r-s)}\Big)\cdot
\Big(\Delta_{m-1}(T_{q,k}^{(s)})\otimes 1\Big).
\nonumber
\end{align}
Then
\begin{align}
&\Delta_m(T_{q,p}^{(r)+}) = \sum_{s=0}^r\sum_{k=1}^{n\over l}\Big(1^{\otimes m-1}\otimes (T_{k,p}^{(r-s)} + (-1)^sT_{-k,p}^{(r-s)})\Big)\cdot
\Big(\Delta_{m-1}(T_{q,k}^{(s)} + T_{-q,k}^{(s)})\otimes 1 \Big).
\nonumber
\end{align}
Hence, using induction on $m$ we have that
\begin{align}\label{MMMM}
&{ev}^{\otimes m}\circ \Delta_m(T_{q,p}^{(r)+}) = \nonumber\\
&\sum_{s=r-1,r}\sum_{k=1}^{n\over l}
\Big(1^{\otimes m-1}\otimes {ev}(T_{k,p}^{(r-s)} + (-1)^sT_{-k,p}^{(r-s)})\Big)\cdot
\Big({ev}^{\otimes m-1}(\Delta_{m-1}(T_{q,k}^{(s)+}))\otimes 1 \Big)
 = \nonumber\\
 &(-1)^r[\sum_{k=1}^{n\over l}
(x^{1}_{p, k}  + (-1)^{r-1}\xi^{1}_{p, k})
\Big(\sum_{1\leq p_2, \ldots, p_{r-1}\leq{n\over l}}
(\sum_{2\leq i_2 < \ldots < i_r\leq m}
(x^{i_2}_{k, p_2}  + (-1)^{r}\xi^{i_2}_{k, p_2})
\ldots
(x^{i_{r}}_{p_{r-1}, q} + \xi^{i_{r}}_{p_{r-1}, q}))\Big) + \nonumber\\
&\sum_{1\leq p_1, p_2, \ldots, p_{r-1}\leq{n\over l}}
\Big(\sum_{2\leq i_1 < \ldots < i_r\leq m}
(x^{i_1}_{p, p_1}  + (-1)^{r+1}\xi^{i_1}_{p, p_1})
(x^{i_2}_{p_1, p_2} + (-1)^{r}\xi^{i_2}_{p_1, p_2})
\ldots
(x^{i_{r}}_{p_{r-1}, q} + \xi^{i_{r}}_{p_{r-1}, q})\Big)] = \nonumber\\
&(-1)^r\sum_{1\leq p_1, p_2, \ldots, p_{r-1}\leq{n\over l}}
\Big(\sum_{1\leq i_1 < i_2 <\ldots < i_r\leq m}
(x^{i_1}_{p, p_1}  + (-1)^{r+1}\xi^{i_1}_{p, p_1})
\ldots
(x^{i_{r}}_{p_{r-1}, q} + \xi^{i_{r}}_{p_{r-1}, q})\Big).
\nonumber
\end{align}
If $m = l$ we obtain the proof of Theorem \ref{TTTT}.
\end{proof}

\begin{corollary}\label{COR1}
Let $ 1\leq p, q\leq {n\over l}$. Then for $r\leq l$
\begin{align}
&\bar{ev}^{\otimes l}\circ \Delta_l(T_{q,p}^{(r)})
= [\sum_{1\leq p_1, p_2, \ldots, p_{r-1}\leq{n\over l}} \nonumber\\
&\sum_{1\leq i_1 < i_2 <\ldots < i_r\leq l}
(x^{i_1}_{p, p_1}  + (-1)^{r+1}\xi^{i_1}_{p, p_1})
(x^{i_2}_{p_1, p_2} + (-1)^{r}\xi^{i_2}_{p_1, p_2})
\ldots
(x^{i_{r}}_{p_{r-1}, q} + \xi^{i_{r}}_{p_{r-1}, q})]_{even},\nonumber\\
&\bar{ev}^{\otimes l}\circ \Delta_l(T_{-q,p}^{(r)})
= [\sum_{1\leq p_1, p_2, \ldots, p_{r-1}\leq{n\over l}} \nonumber\\
&\sum_{1\leq i_1 < i_2 <\ldots < i_r\leq l}
(x^{i_1}_{p, p_1}  + (-1)^{r+1}\xi^{i_1}_{p, p_1})
(x^{i_2}_{p_1, p_2} + (-1)^{r}\xi^{i_2}_{p_1, p_2})
\ldots
(x^{i_{r}}_{p_{r-1}, q} + \xi^{i_{r}}_{p_{r-1}, q})]_{odd}
\nonumber
\end{align}
and $\bar{ev}^{\otimes l}\circ \Delta_l(T_{\pm q,p}^{(r)}) = 0$
for $r > l$.
\end{corollary}
\begin{proof}
Recall that $\bar{ev} = \alpha\circ ev$, where $\alpha$ is the principal anti-automorphism of $U(\g)$: $\alpha(X) = -X$ for all $X\in \g$.
Then
\begin{align}
&\bar{ev}^{\otimes l}\circ \Delta_l(T_{q,p}^{(r)})
= \alpha^{\otimes l}\circ {ev}^{\otimes l}\circ \Delta_l(T_{q,p}^{(r)})
= (-1)^r {ev}^{\otimes l}\circ \Delta_l(T_{q,p}^{(r)}),
\nonumber\\
&\bar{ev}^{\otimes l}\circ \Delta_l(T_{-q,p}^{(r)})
= \alpha^{\otimes l}\circ {ev}^{\otimes l}\circ \Delta_l(T_{-q,p}^{(r)})
= (-1)^r {ev}^{\otimes l}\circ \Delta_l(T_{-q,p}^{(r)}).
\nonumber
\end{align}
\end{proof}

\begin{corollary}\label{COR2}
$$W_{\chi}\cong {ev}^{\otimes l}\circ \Delta_l (Y(Q({n\over l})))$$
\end{corollary}
\begin{proof}
Follows from Corollary \ref{COR} and Corollary \ref{COR1}.
\end{proof}

\begin{definition}
Let $ 1\leq p, q\leq {n\over l}$ and $r>0$. Let
\begin{align}
&{z}_{q,p}^{(r)}
= [\sum_{1\leq p_1, p_2, \ldots, p_{r-1}\leq{n\over l}} \nonumber\\
&\sum_{l\geq i_1\geq i_2\geq\ldots \geq i_r\geq 1}
(x^{i_1}_{p, p_1}  + (-1)^{r+1}\xi^{i_1}_{p, p_1})
(x^{i_2}_{p_1, p_2} + (-1)^{r}\xi^{i_2}_{p_1, p_2})
\ldots
(x^{i_{r}}_{p_{r-1}, q} + \xi^{i_{r}}_{p_{r-1}, q})]_{even},\nonumber\\
&{z}_{-q,p}^{(r)}
= [\sum_{1\leq p_1, p_2, \ldots, p_{r-1}\leq{n\over l}} \nonumber\\
&\sum_{l\geq i_1\geq i_2\geq\ldots \geq i_r\geq 1}
(x^{i_1}_{p, p_1}  + (-1)^{r+1}\xi^{i_1}_{p, p_1})
(x^{i_2}_{p_1, p_2} + (-1)^{r}\xi^{i_2}_{p_1, p_2})
\ldots
(x^{i_{r}}_{p_{r-1}, q} + \xi^{i_{r}}_{p_{r-1}, q})]_{odd}
\nonumber
\end{align}
\begin{align}
&\tilde{z}_{q,p}^{(r)}
= [\sum_{1\leq p_1, p_2, \ldots, p_{r-1}\leq{n\over l}} \nonumber\\
&\sum_{1\leq i_1 < i_2 <\ldots < i_r\leq l}
(x^{i_1}_{p, p_1}  + (-1)^{r+1}\xi^{i_1}_{p, p_1})
(x^{i_2}_{p_1, p_2} + (-1)^{r}\xi^{i_2}_{p_1, p_2})
\ldots
(x^{i_{r}}_{p_{r-1}, q} + \xi^{i_{r}}_{p_{r-1}, q})]_{even},\nonumber\\
&\tilde{z}_{-q,p}^{(r)}
= [\sum_{1\leq p_1, p_2, \ldots, p_{r-1}\leq{n\over l}} \nonumber\\
&\sum_{1\leq i_1 < i_2 <\ldots < i_r\leq l}
(x^{i_1}_{p, p_1}  + (-1)^{r+1}\xi^{i_1}_{p, p_1})
(x^{i_2}_{p_1, p_2} + (-1)^{r}\xi^{i_2}_{p_1, p_2})
\ldots
(x^{i_{r}}_{p_{r-1}, q} + \xi^{i_{r}}_{p_{r-1}, q})]_{odd}
\nonumber
\end{align}
\end{definition}

\begin{theorem}\label{relth}
Let $ 1\leq p, q\leq {n\over l}$ and $r>0$. Then
\begin{equation}\label{rel}
\sum_{t+s = r}[\sum_{j>0}
\Big((-1)^s{z}_{j,p}^{(s)} + (-1)^r{z}_{-j,p}^{(s)}\Big)
\Big(\tilde{z}_{q,j}^{(t)} + \tilde{z}_{-q,j}^{(t)}\Big)] = 0.
\end{equation}
\end{theorem}

\begin{proof}
Note that by (\ref{tilda})
\begin{equation}\label{rel1}
\sum_{t+s = r}[\sum_{j\in\lbrace \pm1, \pm2, \ldots, \pm{n\over l}\rbrace}
(-1)^{(p(q)+p(j))(p(j)+p(p))}
T_{q,j}^{(t)}\tilde{T}_{j,p}^{(s)}]
= \delta_{q, p}.
\end{equation}
Applying the anti-homomorphism
$\bar{ev}^{\otimes l}\circ \Delta_l$ to  (\ref{rel1}), we obtain that
\begin{equation}\label{rel2}
\sum_{t+s = r}[\sum_{j\in\lbrace \pm1, \pm2, \ldots, \pm{n\over l}\rbrace}
\bar{ev}^{\otimes l}\circ \Delta_l(\tilde{T}_{j,p}^{(s)})
\bar{ev}^{\otimes l}\circ \Delta_l(T_{q,j}^{(t)})]
= \delta_{q, p}.\nonumber
\end{equation}
Recall that
$\tilde{T}_{j,p}^{(s)} = S({T}_{j,p}^{(s)})$.
Then by (\ref{last}), if $r>0$, then
\begin{equation}\label{rel2}
\sum_{t+s = r}[\sum_{j\in\lbrace \pm1, \pm2, \ldots, \pm{n\over l}\rbrace}
{U}^{\otimes l}\circ \Delta_l^{op}({T}_{j,p}^{(s)})
\bar{ev}^{\otimes l}\circ \Delta_l(T_{q,j}^{(t)})]
= 0.
\end{equation}
Similarly, we have that
\begin{equation}\label{rel3}
\sum_{t+s = r}[\sum_{j\in\lbrace \pm1, \pm2, \ldots, \pm{n\over l}\rbrace}
{U}^{\otimes l}\circ \Delta_l^{op}({T}_{j,p}^{(s)})
\bar{ev}^{\otimes l}\circ \Delta_l(T_{-q,j}^{(t)})]
= 0.
\end{equation}
Adding the equations (\ref{rel2}) and (\ref{rel3}) we obtain
\begin{equation}\label{rel4}
\sum_{t+s = r}[\sum_{j\in\lbrace 1, 2, \ldots, {n\over l}\rbrace}
{U}^{\otimes l}\circ \Delta_l^{op}
\Big({T}_{j,p}^{(s)} + (-1)^{r-s}{T}_{-j,p}^{(s)}\Big)
\bar{ev}^{\otimes l}\circ \Delta_l
\Big({T}_{q,j}^{(t)} + {T}_{-q,j}^{(t)}\Big)]
= 0.
\end{equation}
Equation (\ref{rel4}) is equivalent to  (\ref{rel}).
\end{proof}

\font\red=cmbsy10
\def\~{\hbox{\red\char'0016}}


\vskip 0.1 in
\section*{Acknowledgments}

\vskip 0.1in
\noindent
This work was supported by a grant from the Simons Foundation (\#354874, Elena Poletaeva) and the NSF grant (\#1303301, Vera Serganova).


\vskip 0.1in


\begin{thebibliography}{10}


\bibitem{BR}
C. Briot, E. Ragoucy,
{$W$-superalgebras as truncations of super-Yangians},
{\it J. Phys. A} {36} (2003), no. 4, 1057--1081.


\bibitem{BBG}
J. Brown, J. Brundan, S. Goodwin,
{Principal
$W$-algebras for $GL(m|n)$}, {\it Algebra Numb. Theory} 7 (2013), 1849--1882.


\bibitem{BK1} J. Brundan, A. Kleshchev,
{Shifted Yangians and finite $W$-algebras}, {\it Adv. Math.} {200} (2006), 136--195.




\bibitem{BK2}
J. Brundan, A. Kleshchev, Representations of shifted Yangians and finite $W$-algebras,
{\it Mem. Amer. Math. Soc.} {196} (2008), no. 918.

\bibitem{K}
 V. G. Kac,
{ Lie superalgebras},
{\it Adv. Math.} {26} (1977) 8--96.


\bibitem{Ko}
B. Kostant,
{On Whittaker vectors and representation theory}, {\it Invent. Math.} {48} (1978) 101--184.

\bibitem{L}
I. Losev,
{Finite $W$-algebras},
{\it Proceedings of the International Congress of Mathematicians.} Volume III, 1281--1307,
Hindustan Book Agency, New Delhi, 2010.
arXiv:1003.5811v1.


\bibitem{N}
M. Nazarov,
{Yangian of the queer Lie superalgebra},
{\it Comm. Math. Phys.} {208} (1999) 195--223.


\bibitem{NS}
M. Nazarov, A. Sergeev,
{Centralizer construction of the Yangian of the queer Lie superalgebra},
Studies in Lie Theory, 417--441, {\it Progr. Math.} {243}, Birkh{\"a}user
Boston, Boston, MA, 2006.


\bibitem{Pe} Y. Peng,
{On shifted super Yangians and a class of finite $W$-superalgebras},
{\it J. Algebra} {422} (2015), 520-562.



\bibitem{PS1}
E. Poletaeva, V. Serganova,
{On finite W-algebras for Lie superalgebras
in the regular case},
In: V. Dobrev (editor) {\it Proceedings of the IX International Workshop ``Lie Theory and Its Applications in Physics''   (Varna, Bulgaria, 20-26 June 2011).}
Springer Proceedings in Mathematics and Statistics, Vol. {36} (2013) 487--497.



\bibitem{PS2}
E. Poletaeva, V. Serganova,
{On Kostant's theorem for the Lie superalgebra $Q(n)$}. {\it Adv. Math.}
 {300} (2016), 320--359.
arXiv:1403.3866v1.


\bibitem{P1}
E. Poletaeva,
{On principal finite W-algebras for the Lie superalgebra $D(2, 1; \alpha)$}.
{\it J. Math. Phys.} {57} (2016), no. 5, 051702.



\bibitem{P2}
E. Poletaeva,
{On finite W-algebras for Lie superalgebras in non-regular case}.
In: V. Dobrev (editor) {\it Proceedings of the XI International Workshop ``Lie Theory and Its Applications in Physics''   (Varna, Bulgaria, 15-21 June 2015).}
Springer Proceedings in Mathematics and Statistics (Springer, Tokyo-Heidelberg)
{191} (2016), 477--488.


\bibitem{Pr1}
A. Premet,
{Special transverse slices and their enveloping algebras},
{\it Adv. Math.} {170} (2002) 1--55.



\bibitem{RS}
E. Ragoucy and P. Sorba,
Yangian realizations from finite $W$-algebras,
{\it Comm. Math. Phys.} {203} (1999) 551--572.

\bibitem{S}
A. Sergeev,
{The centre of enveloping algebra for Lie superalgebra $Q(n, \C)$},
{\it Lett. Math. Phys.} 7 (1983) 177--179.



\bibitem{W}
W. Wang,
{Nilpotent orbits and finite $W$-algebras},
Geometric representation theory and extended affine Lie algebras, 71--105,
{\it Fields Inst. Commun.}  {59},
Amer. Math. Soc., Providence, RI,
2011; arXiv:0912.0689v2.

\bibitem{Z}
L. Zhao,
{Finite $W$-superalgebras for queer Lie superalgebras},
{\it J. Pure Appl. Algebra} {218} (2014)  1184--1194.


\bibitem{ZS}
Y. Zeng and B. Shu,
{Finite $W$-superalgebras for basic Lie superalgebras},
{\it J. Algebra} {438} (2015)  188--234;
arXiv:1404.1150v2.


\end{thebibliography}
\end{document}